\documentclass[12pt]{amsart}
\pdfoutput=1 
\allowdisplaybreaks % for proof reading

\usepackage{amsmath,amsfonts,amssymb}
\usepackage{verbatim}
\usepackage{graphicx}
\usepackage{amsthm}

\newcommand\set[1]{\left\{\sharp1\right\}}

\newtheorem{theorem}{Theorem}[section]
\newtheorem{corollary}[theorem]{Corollary}
\newtheorem{proposition}[theorem]{Proposition}
\newtheorem{lemma}[theorem]{Lemma}

\theoremstyle{remark}
\newtheorem{rmk}{Remark}%[section]
\def \bC {\mathbb C}
\def \bD {\mathbb D}

\def \bH {\mathbb H}

\def \bN {\mathbb N}
\def \bO {\mathbb O}

\def \bR {\mathbb R}

\def \bR {\mathbb R}
\def \bR {\mathbb R}

\def \bZ {\mathbb Z}

\def \cD {\mathcal D}
\def \cE {\mathcal E}
\def \cF {\mathcal F}
\def \cG {\mathcal G}
\def \cH {\mathcal H}

\def \cL {\mathcal L}

\def \cP {\mathcal P}

\def \cS {\mathcal S}

\def \cV {\mathcal V}

\def \fc {\mathfrak c}

\def \fg {\mathfrak g}

\def \fk {\mathfrak k}
\def \fl {\mathfrak l}

\def \fn {\mathfrak n}
\def \fo {\mathfrak o}
\def \fp {\mathfrak p}

\def \fs {\mathfrak s}

\def \fu {\mathfrak u}
\def \fv {\mathfrak v}

\def \fz {\mathfrak z}
\newcommand{\gt}{\mathfrak}

\def \fX {\mathfrak X}

\newcommand{\Spin}{{\rm Spin}}
\def \Sp {\text{\rm Sp}}

\def \su {\mathfrak {su}}

\def \RE {\text{\rm Re}\,}
\def \IM {\text{\rm Im}\,}
\def \al {\alpha}
\def \la {\lambda}
\def \ph {\varphi}
\def \del {\delta}

\def \lan {\langle}
\def \ran {\rangle}
\def \de {\partial}
\def \trans{\,{}^t\!}
\def \half{\frac12}
\def \inv{^{-1}}
\def \rinv\sharp1{^{(\sharp1)}}
\def \supp {\text{\rm supp\,}}

\def \deg {\text{\rm deg\,}}
\def \dim {\text{\rm dim\,}}
\def \span {\text{\rm span\,}}

\def \tr {\text{\rm tr\,}}

\newcommand{\Lie}{{\rm Lie\,}}

\begin{document}

\title[]
{Nilpotent Gelfand pairs\\ and spherical transforms of Schwartz functions\\ II. Taylor expansions on singular sets}

\author[]
{V\'eronique Fischer, Fulvio Ricci, Oksana Yakimova}

\address {King's College London\\ Strand\\ London WC2R 2LS\\ UK}
\email{veronique.fischer@kcl.ac.uk}

\address {Scuola Normale Superiore\\ Piazza dei Cavalieri 7\\ 56126 Pisa\\ Italy}
\email{fricci@sns.it}

\address {Emmy-Noether-Zentrum, Department Mathematik \\
Universit{\"a}t Erlangen-N{\"u}rnberg, Erlangen, Germany}
%\email{yakimova@mccme.ru}
\email{yakimova@mpim-bonn.mpg.de}



\begin{abstract}
This paper is a continuation of \cite{FRY}, in the direction of proving the conjecture that the spherical transform on a nilpotent Gelfand pair $(N,K)$ establishes an isomorphism between the space of $K$-invariant Schwartz functions on $N$ and the space of Schwartz functions restricted to the Gelfand spectrum $\Sigma_\cD$, properly embedded in a Euclidean space.

We prove a result, of independent interest for the representation theoretical problems that are involved, which can be viewed as a generalised Hadamard lemma for $K$-invariant functions on $N$. The context is that of nilpotent Gelfand pairs satisfying Vinberg's condition. This means that the Lie algebra $\fn$ of $N$ (which is step 2) decomposes as $\fv\oplus[\fn,\fn]$ with $\fv$ irreducible under $K$.
\end{abstract}

\maketitle

\begin{flushright}
{\it Dedicated to Joe Wolf}
\end{flushright}

\makeatletter
\renewcommand\l@subsection{\@tocline{2}{0pt}{3pc}{5pc}{}}
\makeatother

\tableofcontents

\section{Outline and formulation of the problem}
\label{sec_intro}

\bigskip

We say that $(N,K)$ is a {\it nilpotent Gelfand pair} (n.G.p. in short) if $N$ is a connected, simply connected nilpotent Lie group, $K$ is a compact group of automorphisms of $N$, and the convolution algebra $L^1(N)^K$ of $K$-invariant integrable functions on $N$ is commutative. This is the same as saying that 
$(K\ltimes N,K)$ is a Gelfand pair %% according to the common terminology.
or that $N$ is a commutative nilmanifold according to \cite[Chapter~13]{W}.

By $\bD(N)^K$ we denote the left-invariant and $K$-invariant differential operators on $N$, and by 
$$
\cD=(D_1,\dots,D_d)\ ,
$$ 
a $d$-tuple of self-adjoint generators of $\bD(N)^K$. To each bounded $K$-spherical function $\ph$ on $N$ we associate (injectively) the $d$-tuple $\xi(\ph)=\big(\xi_1(\ph),\dots,\xi_d(\ph)\big)$ of eigenvalues of $\ph$ as an eigenfunction of $D_1,\dots,D_d$ respectively.

The $d$-tuples $\xi(\ph)$ form a closed subset $\Sigma_\cD$ of $\bR^d$ which is homeomorphic to the Gelfand spectrum $\Sigma$  of $L^1(N)^K$,i.e., the space of bounded spherical functions with the compact-open topology \cite{FeRu}. If $\ph_\xi$ is the spherical function corresponding to $\xi \in\Sigma_\cD$, the spherical transform
\begin{equation}\label{transform}
\cG F(\xi)=\int_N F(x)\ph_\xi(x\inv)\,dx\ ,
\end{equation}
can then be viewed as a function on $\Sigma_\cD$.
\medskip

The following conjecture has been formulated in \cite{FR}.
\medskip

\noindent{\bf Conjecture.} {\it The spherical transform maps the space $\cS(N)^K$ of $K$-invariant Schwartz functions on $N$ isomorphically onto}
$$
\cS(\Sigma_\cD)\overset{\rm def}=\cS(\bR^d)/\{f:f_{|{\Sigma_\cD}}=0\}\ .
$$
\medskip

The inclusion $\cG\big(\cS(N)^K\big)\supseteq \cS(\Sigma_\cD)$ is known to hold in general \cite{ADR2, FR}, so that the conjecture only concerns the opposite inclusion. Moreover, the validity of the conjecture does not depend on the choice of~$\cD$ \cite{FR}.

\medskip

In a nilpotent Gelfand pair $(N,K)$ the group $N$ is at most step-two \cite{BJR90}. We denote by $\fn$ its Lie algebra and by $\fv$ a $K$-invariant complement of the derived algebra $[\fn,\fn]$. We consider $\fn$ endowed with a $K$-invariant scalar product.

We refer the reader to \cite{ADR1, ADR2} for the proof of the conjecture when $N$ is either abelian or the Heisenberg group , and to \cite{FR, FRY} when the following conditions are satisfied :
\begin{enumerate}
\item [(i)] the $K$-orbits in $[\fn,\fn]$ are full spheres, 
\item [(ii)] $K$ acts irreducibly on $\fv$.
\end{enumerate}

In this paper we remove condition (i), still keeping condition (ii). The pairs for which (ii) holds have been classified by E. Vinberg in \cite{V1}, and for this reason we call (ii) {\it Vinberg's condition}. Notice that, under Vinberg's condition, 
$[\fn,\fn]=\fz$, where $\gt z$ is the centre of $\fn$.

We mention here that the classification of nilpotent Gelfand pairs has been completed in \cite{Y1, Y2}, see also \cite[Chapters 13, 15]{W}.

We prove a preliminary result in the direction of proving the conjecture for n.G.p. satisfying Vinberg's condition.  We believe that this result is of independent interest, and its proof requires an interesting combination of methods from noncommutative harmonic analysis and invariant theory. 
The proof of the conjecture for  pairs satisfying Vinberg's condition will appear in \cite{FRY3}.

The proof relies very much on explicit knowledge of the pairs at hand and on the fact that they share some common properties.
Assuming Vinberg's condition and disregarding the pairs considered in \cite{ADR1, ADR2,FR, FRY}, the basic list of n.G.p. to look at is that contained in 
Table~\ref{vinberg}. 
Some explanations are necessary from the beginning:
\begin{enumerate}
\item[(i)] in each case, the Lie bracket $[\ ,\ ]:\fv\times\fv\longmapsto\fz$ is uniquely determined by the requirement of being $K$-equivariant 
(see \cite[Section 13.4B]{W} for explicit expressions);
\item[(ii)] under the action of $K$, the space $\fz$ decomposes as $\fz_0\oplus\check\fz$, where $\check\fz$ is the (possibly trivial) subspace of $K$-fixed elements and $\fz_0$ is its $K$-invariant complement.
\end{enumerate}

\begin{table}[htdp]
\begin{center}
\begin{tabular}{|r||l|l|l|l|l|}
\hline
&$K$&$\fv$&$\fz$&notes&$\fz_0$ (if $\ne\fz$)\\
\hline
1&${\rm SO}_n$&$\bR^n$&$\fs\fo_n$&$n\ge4$&\\
2&${\rm SU}_{2n+1}$&$\bC^{2n+1}$&$\Lambda^2\bC^{2n+1}$&$n\ge2$& \\
3&$\text{Sp}_2\times\text{Sp}_n$&$\bH^2\otimes\bH^n$&$\fs\fp_2$&$n\ge2$&\\
\hline
4&${\rm U}_{2n+1}$&$\bC^{2n+1}$&$\Lambda^2\bC^{2n+1}\oplus\bR$&$n\ge1$&$\Lambda^2\bC^{2n+1}$ \\
5& ${\rm SU}_{2n}$& $\bC^{2n}$& $\Lambda^2\bC^{2n}\oplus\bR$&$n\ge2$& $\Lambda^2\bC^{2n}$\\
6&${\rm U}_n$&$\bC^n$&$\fu_n$&$n\ge2$&$\su_n$\\
7&$\text{Sp}_n$&$\bH^n$&$HS^2_0\bH^n\oplus\IM\bH$&$n\ge2$&$HS^2_0\bH^n$ \\
\hline
8&${\rm U}_2\times {\rm SU}_n$&$\bC^2\otimes\bC^n$&$ \fu_2$& $n\ge2$&$\su_2$\\
9&${\rm U}_2\times\text{Sp}_n$&$\bC^2\otimes\bH^n$&$ \fu_2$ &$n\ge2$&$\su_2$\\
10&${\rm U}_1\times\text{Spin}_7$&$\bC\otimes\bO$&$\IM\bO\oplus\bR$&&$\IM\bO$\\
\hline
\end{tabular}
\end{center}
\bigskip
\caption{}%{The list of Gelfand pairs}
\label{vinberg}
\end{table}

All other nilpotent Gelfand pairs satisfying Vinberg's condition are obtained from those in Table \ref{vinberg} by either of the following operations:
\begin{enumerate}
\item[(a)] normal extensions of $K$: replace $K$ by a larger group $K^\#$ of automorphisms of $N$ with $K\triangleleft K^\#$;
\item[(b)] central reductions: if $\fz$ has a nontrivial proper $K$-invariant subspace $\fs$, replace $\fn$ by $\fn/\fs$.
\end{enumerate}

In \cite{FRY} we proved that if $N,K,K^\#$ are as in (a), and the conjecture is true for $(N,K)$, then it is also true for $(N,K^\#)$. It will be proved in \cite{FRY3} that, applying a central reduction to a pair for which the conjecture is true, the resulting pair also satisfies the conjecture. We will therefore concentrate our attention on the pairs in Table~\ref{vinberg}.
\medskip

\medskip

In order to formulate our main result, Theorem~\ref{main} below, we need to describe some aspects of the structure of the Gelfand spectrum $\Sigma$ and the way they reflect on its embedded copy $\Sigma_\cD$ in $\bR^d$.

In $\Sigma$ we distinguish a relatively open and dense ``regular set'' from a ``singular set'', and singular points may have different levels of singularity. Since all bounded spherical functions are of positive type \cite{FRY}, a bounded spherical function on $N$ can be expressed as an average over $K$ of matrix entries of some irreducible unitary representation of $N$. Hence we can associate to each bounded spherical function a $K$-orbit of characters of $\exp\fz$. 

The regular elements of $\Sigma$ are those associated to orbits of maximal dimension.  
Among singular points, the highest level of singularity is reached by the bounded spherical functions associated to the characters of $\exp\fz$ which are fixed by all of $K$, i.e., which are trivial on $\exp\fz_0$. These are the spherical functions which factor to the quotient group $\check N=N/\exp_N\fz_0$. We call $\Sigma^0$ the subset of $\Sigma$ consisting of these spherical functions, and $\Sigma^0_\cD$ the corresponding subset of $\Sigma_\cD$.

At this point it is convenient to introduce a preferred system $\cD$ of generators of $\bD(N)^K$, obtained, via symmetrisation, from the bases of fundamental $K$-invariants on $\fn$  listed, case by case, in Section 7 of \cite{FRY}. We denote by $\rho=(\rho_1,\dots,\rho_d)$ the $d$-tuple of these polynomials.

The polynomials $\rho_j$ have the property of being homogeneous in each of the variables $v\in\fv$, $z\in\fz_0$, $t\in\check\fz$. 
For each $j$, we denote by $[j]$ the degree of $\rho_j$ in the $\fz_0$-variables. 

Notice that $[j]>0$ if and only if $D_j$ annihilates all the spherical functions which factor to $\check N$. At the same time, the polynomials $\rho_j$ with $[j]=0$ provide a system of fundamental $K$-invariants on the Lie algebra of $\check N$, $\check\fn\cong\fv\oplus\check\fz$, where $[\check\fn,\check\fn]=\check\fz$. Symmetrising $\rho_j$ on $\check N$ produces an operator $\check D_j\in\bD(\check N)^K$, which is the push-forward of $D_j$ via the canonical projection.

\medskip

Suppose that the $D_j\in\cD$ have been ordered so that $D_1,\dots, D_{d_0}$ are the operators with $[j]=0$. Then $\Sigma_\cD^0$ can be realised as the intersection of $\Sigma_\cD$ with the coordinate subspace
$$
\Sigma_\cD^0=\{\xi\in\Sigma_\cD:\xi_{d_0+1}=\cdots=\xi_d=0\}\ .
$$

What has been said above shows that there is a natural identification of $\Sigma_\cD^0$ with the Gelfand spectrum $\Sigma_{{\check D}}$ of the pair $(\check N,K)$, with ${\check D}=\{\check D_1,\dots,\check D_{d_0}\}$. 

We will decompose the variables of $\bR^d$ as
$\xi=(\xi',\xi'')$, with $\xi'=(\xi_1,\dots,\xi_{d_0})$,  $\xi''=(\xi_{d_0+1},\dots,\xi_d)$. To have a consistent notation, multi-indices $\al''$ will have components indexed from $d_0+1$ to $d$, so that monomials $\xi^{\al''}$ only depend on $\xi''$ and, similarly,
$$
D^{\al''}=D_{d_0+1}^{\al_{d_0+1}}\cdots D_{d}^{\al_d}\ .
$$

We set $[\al'']=\sum_{j=d_0+1}^d[j]\al_j$. Of course, $[\al'']$ equals the order of derivation of $D^{\al''}$ in the $\fz_0$-variables.

\medskip

Let us go back to the conjecture. Given a function $F\in\cS(N)^K$ we are interested in proving that its spherical transform \eqref{transform} extends from $\Sigma_\cD$ to a Schwartz function on $\bR^d$. In \cite{FRY}, one of the crucial points in the proof was Proposition 5.1, providing a Taylor development of $\cG F$ along the singular set;
in that situation, there was just one level of singularity. 

Recast in our present situation, that result can be phrased as follows: given $k\in \bN$,
there exist $K$-invariant Schwartz functions $\{F_{\al''}\}_{[\al'']\le k-1}$ on $N$, with $\cG F_{\al''}$ only depending on $\xi'$, and such that
\begin{equation}\label{expansion}
F=\sum_{[\al'']\le k-1}D^{\al''} F_{\al''}
+\sum_{|\beta|=k}\de_z^\beta R_\beta\ ,
\end{equation}
with $R_\beta\in \cS(N)$ for every $\beta$.

It is clear, by induction, that it will be sufficient to show that the remainder term
$$
\Phi_k(v,z,t)=\sum_{|\beta|=k}\de_z^\beta R_\beta(v,z,t)
$$
can be further expanded as
\begin{equation}
\label{expansion1}
\Phi_k=\sum_{[\al'']=k}D^{\al''} F_{\al''}+\sum_{|\gamma|=k+1}\de_z^\gamma S_\gamma
\end{equation}
for some new functions $F_{\al''}\in \cS(N)^K$, $[\al'']=k$, with $\cG F_{\al''}$ only depending on $\xi'$, and some new $S_\gamma\in \cS (N)$.

Formula \eqref{expansion1} can be seen as a noncommutative Hadamard-type formula. Its simplest abelian relative is the statement that if a radial Schwartz function on $\bR^n$ is a sum of second-order derivatives of Schwartz functions, then it is the Laplacian of a radial Schwartz function (see also Section \ref{sec_N'ab}).

We now give the argument that allows to reduce the proof of \eqref{expansion1} to proving Theorem \ref{main}.

 It is convenient to introduce modified versions of the operators $D_j$, an operation that corresponds to replacing the group $N$ with the direct product $\tilde N=\check N\times\fz_0$ of $\check N$ and the additive group $\fz_0$. We remark that $(\tilde N,K)$ is also a Gelfand pair (not satisfying Vinberg's condition), as it can be checked from the classification in \cite{Y2} or, through a direct argument, from the fact that the Lie algebra $\tilde\fn$ is a contraction of $\fn$.
 
From the same system of invariants $\rho_j$ used to generate the differential operators $D_j$ on $N$, we produce, by symmetrisation on $\tilde N$, a system $\tilde \cD=\{\tilde D_1,\dots,\tilde D_d\}$ of generators of $\bD(\tilde N)^K$. We also use the same coordinates $(v,z,t)\in\fv\times\fz_0\times\check\fz$ on $\tilde N$, via the exponential map $\exp_{\tilde N}$.
Taking advantage of this common coordinate system for $N$ and $\tilde N$, we can compare $D_j$ and $\tilde D_j$ as follows: the left-invariant vector field corresponding to the basis element $e_\nu\in\fv$ is
$$
X_\nu=\de_{v_\nu}+\sum_i b_i(v)\de_{z_i}+\sum_\ell c_\ell(v)\de_{t_\ell}
$$
on $N$, and
$$
\tilde X_\nu=\de_{v_\nu}+\sum_\ell c_\ell(v)\de_{t_\ell}
$$
on $\tilde N$.
Therefore,  
$$
D_j-\tilde D_j=\sum_{\al,\beta,\gamma\,:\,|\beta|\ge1} a^j_{\al,\beta,\gamma}(v)\de_v^\al\de_z^\beta\de_t^\gamma\ ,
$$
where each term contains at least one derivative in the $z$-variables.

This implies that, if $[\al'']=k$, then each term in $D^{\al''}-\tilde D^{\al''}$ contains at least $k+1$ derivatives in $z$. 
Then it will be sufficient to prove \eqref{expansion1} with each $D^{\al''}$ replaced by $\tilde D^{\al''}$, 
since the difference can be absorbed in the remainder term. 
Therefore \eqref{expansion1} is equivalent to
\begin{equation}\label{expansion2}
\Phi_k=\sum_{[\al'']= k}
\tilde D^{\al''} 
F_{\al''}
+\sum_{|\gamma|=k+1}\de_z^\gamma S_\gamma\ .
\end{equation}

To both sides of \eqref{expansion2} we apply Fourier transform in the $z$-variables, that we denote by ``$\,\widehat{\phantom a}\,$'', e.g.,
$$
\widehat {F_{\al''}}(v,\zeta,t)=\int_{\fz_0}F_{\al''}(v,z,t)e^{-i\lan z,\zeta\ran}\,dz\ ,
$$
where $\lan\ ,\ \ran$ is the given $K$-invariant scalar product on $\fz_0$.
We obtain:
\begin{equation}\label{expansion3}
\widehat \Phi_k(v,\zeta,t)=\sum_{[\al'']= k}
\tilde D_\zeta ^{\al''} 
\widehat{F_{\al''}}(v,\zeta,t)
+\sum_{|\gamma|=k+1}(i\zeta)^\gamma \widehat{S_\gamma}(v,\zeta,t)\ ,
\end{equation}
where each $\tilde D_{j,\zeta}$ is obtained from $\tilde D_j$ by replacing each derivative $\de_{z_\ell}$ by $i\zeta_\ell$.

Modulo error terms that involve higher-order powers of $\zeta$, we are left with proving that 
the $k$-th order term in the Taylor expansion in $\zeta$ of $\widehat{\Phi_k}(v,\zeta,t)$, i.e.,
$$
\sum_{|\gamma|=k}\frac{\zeta^\gamma}{\gamma!}\de_z^\beta \widehat {R_\beta}(v,0,t)
\ ,
$$
equals the $k$-th order term in the Taylor expansion in $\zeta$ of \eqref{expansion3}, i.e.,
$$
\sum_{[\al'']= k}
\tilde D_\zeta ^{\al''}\widehat{F_{\al''}}(v,0,t)\ .
$$

This equality is the subject of our main theorem.
\begin{theorem}
\label{main}
 Let $G$ be a $K$-invariant function on $\tilde N$ of the form
\begin{equation}\label{G}
 G(v,\zeta,t)=\sum_{|\gamma|=k}\zeta^\gamma G_\gamma(v,t)\ ,
\end{equation}
 with $G_\gamma\in\cS(\check N)$. 
 Then there are $H_{\al''}\in\cS(\check N)^K$, for $[\al'']=k$, such that
\begin{equation}\label{Hal}
 G=\sum_{[\al'']= k}\tilde D_\zeta^{\al''} H_{\al''}\ .
\end{equation}

More precisely, given a Schwartz norm $\|\ \|_{(p)}$, the functions $H_{\al''}$ can be found so that, for some $q=q(k,p)$, $\|H_{\al''}\|_{(p)}\le C_{k,p}\sum_{|\gamma|=k}\|G_\gamma\|_{(q)}$, for every $\al''$, $[\al'']=k$.
\end{theorem}
\medskip

In Section \ref{sec_N'ab}, we prove Theorem \ref{main} for the pairs in the first block of Table \ref{vinberg}.
Indeed, in these cases the group $\check N$ is reduced to $\fv$ and is abelian.

The rest of the article will be devoted to the proof of Theorem \ref{main}  for the other pairs, where
$\check N$ is a Heisenberg group, with the exception of line 7, where it is a ``quaternionic Heisenberg group''
with Lie algebra $\bH^n\oplus \IM\bH$. 

In Section \ref{sec_polynomials}, we develop a careful analysis of the structure of the $K$-invariant polynomials on $\fv\oplus\fz_0$, describing the $K$-invariant irreducible subspaces of the symmetric algebras over $\fv$ and $\fz$ that are involved.

In Section \ref{sec_fourier}, we reduce the proof of Theorem \ref{main} to an equivalent problem of representing vector-valued $K$-equivariant functions in terms of $K$-equivariant differential operators applied to $K$-invariant scalar functions (Proposition \ref{Vm-operator}). Then we analyse the images of these differential operators in the Bargmann representations of $\check N$, identifying the $K$-invariant irreducible subspaces of the Fock space on which they vanish. This analysis reveals interesting connections between these operators and the natural action of $K$ itself on the Fock space, once both are realised to be part of the metaplectic representation.

Finally, in Sections \ref{sec_N'nonab} and \ref{sec_conclusion}, we complete the proof of Theorem \ref{main} for the pairs with $\check N$ nonabelian.

\vskip1cm

\section{Proof of Theorem \ref{main} for $\check N$ abelian}\label{sec_N'ab}

\bigskip

In this section, we consider the pairs in the first block of Table \ref{vinberg}, where $\fz_0=\fz$ and, therefore, $\check N=\fv$ is abelian. We call $(\fv,K)$ an {\it abelian pair}.
In this case,  one can prove that Theorem \ref{main} is true 
with the additional property that the functions $H_\alpha$ can be chosen independently of $p$ and depending linearly on the~$G_\gamma$. 

Via Fourier transform in $\fv$,
this statement is equivalent to the Proposition \ref{prop_cq_SMH} below.
We first explain the notation.
We split the set $\rho$ of fundamental invariants into the two subsets $\rho'$, $\rho''$, where $\rho'$ contains the polynomials depending only on $v\in\fv$, and $\rho''$ those which contain $z\in\fz$ at a positive power. This notation matches with the splitting of coordinates $(\xi',\xi'')$ on the Gelfand spectrum introduced in Section \ref{sec_intro}.

\begin{proposition}
\label{prop_cq_SMH}
Let $G\in C^\infty(N)^K$ satisfying
$$
G(v,\zeta)=\sum_{|\gamma|=k} \zeta^\gamma G_\gamma(v)
\ ,
$$
with $G_\gamma\in \cS (\fv)$.
Then there exist $g_{\alpha''}\in\cS(\bR^{d_0})$, $[\alpha'']=k$, depending linearly and continuously on $\{G_\gamma\}_\gamma$ and
such that
$$
G(v,\zeta)=\sum_{[\alpha'']=k} \rho(v,\zeta)^{\alpha''}  g_{\alpha''}\circ \rho'(v)\ .
$$
\end{proposition}

The proof is quite simple and relies on two adapted versions of Hadamard's Lemma on one side, and of the Schwarz-Mather theorem \cite{Mather, Schw} on the other side.
Hadamard's Lemma states that if a function of two variables $f(x,y)\in C^\infty(\bR^n\times\bR^m)$ satisfies  $f(0,y)=0$ for every $y$, then there exist $C^\infty$-functions $g_j(x,y)$, $j=1,\dots,n$, such that
$$
f(x,y)=\sum_{j=1}^n x_jg_j(x,y)\ .
$$

Adapting  the proof of Hadamard's lemma given in Proposition 5.3 in \cite{FRY},
it is easy to show the following.

\begin{lemma}\label{hadamard}
\quad
Let $f(x,y)\in C^\infty(\bR^n\times\bR^m)$ and $k\in \bN$. Then there exist smooth functions $g_\alpha(y)\in C^\infty(\bR^m)$, $|\alpha|\leq k$,
and  $R_\alpha(x,y)\in C^\infty(\bR^n\times\bR^m)$, $|\alpha|= k+1$,
 such that
$$
f(x,y)=\sum_{|\alpha|\leq k} x^\alpha g_\alpha (y)
+ \sum_{|\alpha|=k+1} x^\alpha R_\alpha (x,y)\ .
$$

Furthermore if $f(x,y)\in C^\infty(\bR^n)\hat \otimes\cS(\bR^m)$
in the sense that, for every $L$,
$$
\sup_{\begin{matrix}
|\alpha|,|\beta|,|x|\leq L\\ y\in \bR^n\end{matrix}
}
(1+|y|)^L |\partial_x^\alpha\partial_y^\beta f(x,y)|
<\infty
\ ,
$$
then the functions $g_\alpha(y)$, $|\alpha|\leq k$,
and  $R_\alpha(x,y)$, $|\alpha|= k+1$,
 can be chosen in $\cS(\bR^m)$ and $C^\infty(\bR^n)\hat \otimes\cS(\bR^m)$
 respectively,
and depending linearly and continuously on $f$.
\end{lemma}

\begin{proof}[Proof of Proposition \ref{prop_cq_SMH}]
All the polynomials $\rho_j$ are homogeneous in $v$ and $z$
and, for $j=1,\ldots, d_0$, they only depend on $v$.
Hence it is easy to adapt the proof of Theorem 6.1 in \cite{ADR2}
to show that there exists 
a continuous linear operator
$\tilde\cE: \big(\cS(\fv)\hat\otimes C^\infty(\fz)\big)^K\rightarrow \cS(\bR^{d_0})\hat\otimes C^\infty(\bR^{d-d_0})$
such that 
$\tilde\cE(g)\circ \rho =g$ for every $g\in \big(\cS(\fv)\hat\otimes C^\infty(\fz)\big)^K$.
So let $h=\tilde \cE (G)$. 
Using Lemma \ref{hadamard}, we obtain that,
for any $\xi=(\xi',\xi")\in \bR^d$,
$$
h(\xi)
=
\sum_{|\alpha''|\leq k}\xi^{\al''}  g_{\alpha''}(\xi')
+
\sum_{|\alpha''|= k+1} \xi^{\al''}  U_{\alpha''}(\xi)
\ ,
$$
where each $g_{\alpha''}$ depends linearly and continuously 
on $h\in \cS(\bR^{d_0})\hat\otimes C^\infty (\bR^{d-d_0})$,
hence on $\{G_\gamma\}_\gamma$.
Composing with $\rho$, we get:
$$
G(v,\zeta)
=
h\circ \rho(v,\zeta)
=
\sum_{|\alpha''|\leq k}
\rho(v,\zeta)^{\alpha''} g_{\alpha''}\big(\rho'(v)\big)
+
\sum_{|\alpha''|=k+1}
\rho(v,\zeta)^{\alpha''} U_{\alpha''}\big(\rho(v,\zeta)\big)
\ .
$$

As $G$ is a polynomial of degree $k$ in $\zeta$, we have:
$$
G(v,\zeta)
=
\sum_{[\alpha'']= k}
\rho(v,\zeta)^{\alpha''} g_{\alpha''}\big(\rho'(v)\big)
\ .\qedhere
$$
\end{proof}

  \vskip1cm
 
 \section{$\check N$ nonabelian: structure of $K$-invariant polynomials on $\fv\oplus\fz_0$} \label{sec_polynomials}

\bigskip
 
 From Table~\ref{vinberg} we isolate the last two blocks, 
 i.e., the cases where $\check N$ is not abelian. 
 To each line we add the list of fundamental $K$-invariants on $\fv\oplus \fz_0$ as it appears in Theorem~7.5 of \cite{FRY}. 
 We split the set $\rho$ of these invariants into three subsets, $\rho_\fv$, $\rho_{\fz_0}$, $\rho_{\fv,\fz_0}$, 
 containing the polynomials which depend, respectively, only on $v\in\fv$, only on $z\in\fz_0$, or on both $v$ and $z$. 
We call the last ones the ``mixed invariants''. 
It follows from \cite[Corollary 7.6]{FRY} that the algebra $\cP(\fv\oplus\fz_0)^K$ is freely generated by $\rho=\rho_\fv\cup\rho_{\fz_0}\cup\rho_{\fv,\fz_0}$. 
We convene to use the letters $r,q,p$ to denote, respectively, elements of $\rho_\fv$, $\rho_{\fz_0}$, $\rho_{\fv,\fz_0}$.
 
The result is Table~\ref{invariants}.
 Note that expressions like $z^k$ refer to the $k$-th power of a matrix $z$. 
As in  \cite{FRY}, at lines 9 and 10,
we decompose $\fv$ as the sum of two 
subspaces invariant under $\Sp_1{\times}\Sp_n$ and $\Spin_7$ respectively:
$$
\bC^2\otimes\bC^{2n}=\bR^{4n}{\oplus}\bR^{4n}
\quad
\mbox{and}\quad \mathbb C\otimes\bO= \bR^8{\oplus}\bR^8\ ,
$$
and we write an element $v$ of $\fv$ as $v=x+iy$ and
$v=v_1+iv_2$ accordingly.
For line 10, we also identify $\bR^8$ with $\bO$ and the conjugation there 
is the octonian conjugation.

 \begin{table}[htdp]
\begin{center}
\begin{tabular}{|r||l|l|l||c|c|c|}
\hline
&$K$&$\fv$&$\fz_0$&$r_k(v)$&$q_k(z)$&$p_k(v,z)$\\
\hline
\hline
4&${\rm U}_{2n+1}$&$\bC^{2n+1}$&$\Lambda^2\bC^{2n+1}$&$|v|^2$&$\begin{matrix}\tr\big((\bar zz)^k\big)\\ (1\le k\le n)\end{matrix}$&$\begin{matrix}v^*(\bar zz)^kv\\ (1\le k\le n)\end{matrix}$ \\
\hline
5& ${\rm SU}_{2n}$& $\bC^{2n}$&$\Lambda^2\bC^{2n}$&$|v|^2$&$\begin{matrix}\tr\big((\bar zz)^k\big)\\ 
\mbox{{\small ($1\le k\le n{-}1$)}}\\ {\rm Pf}(z)\,,\, \overline{{\rm Pf}(z)}\end{matrix}$&$\begin{matrix}v^*(\bar zz)^kv\\ 
\mbox{{\small($1\le k\le n{-}1$)}}\end{matrix}$\\
\hline
6&${\rm U}_n$&$\bC^n$&$\su_n$&$|v|^2$&$\begin{matrix}\tr\big((iz)^k\big)\\ (2\le k\le n)\end{matrix}$&$\begin{matrix}v^*(iz)^kv\\ \mbox{{\small($1\le k\le n{-}1$)}}\end{matrix}$\\
\hline
7&$\text{Sp}_n$&$\bH^n$&$HS^2_0\bH^n$ &$|v|^2$&$\begin{matrix}\tr z^k\\ (2\le k\le n)\end{matrix}$&$\begin{matrix}v^*z^kv\\ \mbox{{\small($1\le k\le n{-}1$)}}\end{matrix}$\\
\hline
\hline
8&${\rm U}_2{\times} {\rm SU}_n$&$\bC^2{\otimes}\bC^n$&$\su_2$&$\begin{matrix}\tr\big((vv^*)^k\big)\\  (k=1,2)\end{matrix}$&$|z|^2$&$i\tr(v^*zv)$\\
\hline
9&${\rm U}_2{\times}\text{Sp}_n$&$\bC^2{\otimes}\bC^{2n}$&$\su_2$&$\begin{matrix}\tr\big((vv^*)^k\big)\\  (k=1,2)\\ |x|^2|y|^2-(\trans xy)^2\end{matrix}$&$|z|^2$&$i\tr(v^*zv)$\\
\hline
10&${\rm U}_1{\times}\text{Spin}_7$&$\bC{\otimes}\bO$&$\IM\bO$&$\begin{matrix}|v|^2\\ |v_1|^2|v_2|^2-\big(\RE(v_1\bar v_2)\big)^2\end{matrix}$&$|z|^2$&$\RE\big(z(v_1\bar v_2)\big)$\\
\hline
\end{tabular}
\end{center}
\bigskip
\caption{}%{Invariants for $\check N$ nonabelian}
\label{invariants}
\end{table}

\medskip

If $X$ is a real vector space, 
we call $\cP(X)$ the polynomial algebra over $X$, 
and $\cP^k(X)$ the subspace of homogeneous polynomials of degree $k$. 
When $X$ is endowed by a complex structure, 
we denote by $\cP^{k_1,k_2}(X)$ the terms in the splitting of $\cP(X)$ according to bi-degrees; 
for example $\cP^{k,0}$ is the space of holomorphic polynomials in $\cP^k$. 

This applies in particular to $\fv$, which always carries a complex structure, and to $\fz_0$ at lines 4 and 5. At line 7,  in fact, $\fv$ admits a different complex structure for every choice of a unit quaternion.

The indexing of the elements $p_k(v,z)$ of $\rho_{\fv,\fz_0}$  is assumed to match with the notation of Table~\ref{invariants} when there is more than one element in the family. 

Coherently with the notation used in the previous sections, 
if $p^\al(v,z)$ is a monomial in the $p_k$, 
we denote by $|\al|$ the usual length of the multi-index $\alpha$, 
and by $[\al]$ the degree of the polynomial $p^\al(v,z)$
in $z$. 
When $\fz_0$ is a complex space, we denote by $[\![\al]\!]$ the bi-degree of $p^\al$ in $z,\bar z$. The same convention on the use of  $[\ ]$ and 
$[\![\ ]\!]$ applies to monomials in the $q_k$.

\medskip

The pairs in Table~\ref{invariants} are distinguished by two properties. 
The first is  that 
we can add a subspace $\check\fz$, of dimension one or three, to $\fz_0$ keeping  
$(K,N)$ as a nilpotent Gelfand pair. The second is that 
$\gt v{\oplus}\check\fz$, regarded as a quotient of $\gt n$, 
is either a Heisenberg Lie 
algebra or a quaternionic Heisenberg Lie algebra. 
Another observation will be of particular importance in the future. 

\begin{rmk}\label{stabiliser}
 Fix $\zeta\in\fz_0$ and let $K_\zeta$ be its stabiliser in $K$. 
 Then the pair $(\check N,K_\zeta)$ is also a nilpotent Gelfand pair. 
The result goes back to 
Carcano's characterisation of nilpotent Gelfand pairs in terms of multiplicity free actions \cite{C}. 
%% and deduce directly multiplicity freeness for $(N,K_\zeta)$, with $\zeta$ regular, from multiplicity freeness for $(N,K)$.  
An alternative proof can be found, e.g., in \cite[Ch.2,\S4]{V1}.
\end{rmk}

The first dividend we get is the following.  
Evaluating $K$-invariants at $\zeta\in\gt z_0$, 
considered as a point of $\gt z$, 
we get $K_\zeta$-invariant polynomials 
on $\gt v{\times}\{\zeta\}$, or better to say on $\gt v$. 
These polynomials have the same degree in $v$ and in $\bar v$ 
\cite[Section~4]{HoweUm}.
Hence 
the expressions of the polynomials $r_k(v)$ and $p_k(v,z)$ 
must also have the same degree in $v$ and $\bar v$
(this can be seen directly from Table~\ref{invariants}).
Therefore we have the splitting
$$
\cP(\fv\oplus\fz_0)^K=\sum_{m,k\ge0} \big(\cP^{m,m}(\fv)\otimes\cP^k(\fz_0)\big)^K\ .
$$

We want to refine this decomposition, by putting special attention on the mixed invariants.
Any mixed invariant $p(v,z)$ in $\big(\cP^{m,m}(\fv)\otimes\cP^k(\fz_0)\big)^K$ can be expanded as 
\begin{equation}\label{irreducibles}
p=\sum_j p_{V_j,W_j}\ ,
\end{equation}
where, for each $j$, $V_j$ and $W_j$ are $K$-invariant, irreducible subspaces of 
$\cP^{m,m}(\fv)$ and $\cP^k(\fz_0)$ respectively, with $V\sim W$ 
equivalent to $W$ as a $K$-module,
and
\begin{equation}\label{dualbases}
p_{V_j,W_j}(v,z)=\sum_h a_h(v)b_h(z)\ ,
\end{equation}
with $\{a_h\}$ and $\{b_h\}$ being orthonormal dual 
bases. %%%  modulo the identification of $V$ with $W'$.

In a rather canonical way, we will now replace the basis of monomials $p^\al(v,z)q^\beta(z)r^\gamma(v)$ by a new basis, obtained by replacing each $p^\al$ by a new polynomial $\widetilde{p^\al}$ which is ``irreducible'', in the sense that it equals $p_{V_\al,W_\al}$ for appropriate irreducible $V_\al,W_\al$.

\medskip

Before going into this construction, we remark some useful aspects of the list of pairs and invariants in Table \ref{invariants}.

\medskip
\begin{rmk}\label{tablecomments}
\quad

\begin{enumerate}
\item[\rm(a)] The first block of Table \ref{invariants} contains four infinite families, with both $\dim\fv$ and $\dim\fz_0$ increasing with the parameter $n$. Each pair admits a single invariant in $\rho_\fv$, and several in $\rho_{\fz_0}$ and $\rho_{\fv,\fz_0}$.
\item[\rm(b)] Inside the first block, the pairs at lines 4 and 5 have a special feature, in that $\fn_0$ is a complex Lie algebra 
and $\fz_0$ is a complex space.
The invariants for a pair in line 4 or 5 can be associated  with the lower degrees invariants for the pair at line 6 with the same $\fv$,
and
the former invariants coincide with those of the latter but evaluated at $(v,-i\bar zz)$ instead of $(v,z)$.
\item[\rm(c)] Each line in the second block contains either an ``exceptional'' isolated pair (line 10), 
or an infinite family (lines 8, 9), but with $\fz_0$ fixed. Each pair admits a single invariant in $\rho_{\fz_0}$ and in $\rho_{\fv,\fz_0}$, but several in $\rho_\fv$.
\item[\rm(d)] For each pair, the $k$-th mixed polynomial $p_k(v,z)$ is a finite sum 
\begin{equation}\label{pk}
p_k(v,z)=\sum_{j=1}^{\nu_1} \ell_j(v)b_{jk}(z)\ ,
\end{equation}
with $\nu_1=\dim\fz_0$ and the $\ell_j$ independent of $k$. 
\item[\rm(e)] For the pairs at lines 6-10, the polynomials $b_{j1}(z)$ appearing in the expression \eqref{pk} of $p_1$ are the coordinate functions on $\fz_0$.
The %complex 
real span of the polynomials $\ell_j(v)$ is a $K$-invariant subspace of $\cP^{1,1}(\fv)$ equivalent to $\fz_0$. %%%% $\fz_0^\bC$.
\item[\rm(f)] At lines 6, 7 and for $k>1$, $p_k(v,z)$ (resp. $q_k(z)$) equals, up to a power of $i$, $p_1(v,z^k)$ (resp. $q_1(z^k)$). Here again $z^k$ is the $k$-th power of a matrix. 
\end{enumerate}
\end{rmk}

\medskip

%%%At this point,
Because of Remark \ref{tablecomments}\,(b),
 we first restrict our attention to the pairs of lines 6-10.

For given $m,k$, we look at the structure of $\big(\cP^{m,m}(\fv)\otimes\cP^k(\fz_0)\big)^K$, the space of $K$-invariant polynomials on $\fv\oplus\fz_0$ of bi-degree $(m,m)$ in $v$ and degree $k$ in $z$. 

Inside $\cP^{m,m}(\fv)$ consider the subspace generated by polynomials which are divisible by elements of $\rho_\fv$, and let $\cH^{m,m}(\fv)$ its orthogonal complement. More explicitly, if $r^\gamma(v)$ is a monomial in the $r_j$ of bi-degree $(\del_\gamma,\del_\gamma)$, then
$$
\cH^{m,m}(\fv)=\Big(\sum_{1\le\del_\gamma\le m}r^\gamma\cP^{m-\del_\gamma,m-\del_\gamma}(\fv)\Big)^\perp\ .
$$

With an abuse of language, we call $\cH^{m,m}(\fv)$ the {\it harmonic subspace} of $\cP^{m,m}(\fv)$.
By the $K$-invariance of each $\cH^{m,m}(\fv)$,
$$
\big(\cP^{m,m}(\fv)\otimes\cP^k(\fz_0)\big)^K={\sum_{0\le \del\le m}}^{\!\!\oplus}{\sum_{\del_\gamma=\del}}^{\!\!\oplus}r^\gamma\big(\cH^{m-\del_\gamma,m-\del_\gamma}(\fv)\otimes\cP^k(\fz_0)\big)^K\ .
$$

Similarly, we set
$$
\cH^k(\fz_0)=\Big(\sum_{1\le [\beta]\le k}q^\beta\cP^{k-[\beta]}(\fz_0)\Big)^\perp\ .
$$

For an element $p$ of $\big(\cP^{m,m}(\fv)\otimes\cP^k(\fz_0)\big)^K$ we denote by $\tilde p$ its $\fv$-{\it harmonic component}, i.e., its component in $\big(\cH^{m,m}(\fv)\otimes\cP^k(\fz_0)\big)^K$.

Finally, we denote by $\cP^m(\ell)\subset \cP^{m,m}(\fv)$ the space generated by the monomials of degree $m$ in the $\ell_j$.

\begin{proposition}\label{structure}
Let $K$, $\fv$, $\fz_0$ be as in Table \ref{invariants}, lines 6-10.
\begin{enumerate}
\item[\rm (i)] If $k<m$, $\big(\cH^{m,m}(\fv)\otimes\cP^k(\fz_0)\big)^K$ is trivial.
\item[\rm (ii)] For $k=m$, $\big(\cH^{m,m}(\fv)\otimes\cP^m(\fz_0)\big)^K$ is one-dimensional, and it is generated by $\widetilde{p_1^m}$.
\item[\rm (iii)] Let $V_m=\cH^{m,m}(\fv)\cap\cP^m(\ell)$.
Then $V_m$ is absolutely irreducible, i.e., it stays irreducible 
as a representation of $K^{\mathbb C}$ after the complexification 
$V_m\otimes_{\mathbb R}\mathbb C$.
We fix an orthonormal basis $a^{(m)}_j$, $1\le j\le \nu_m$, of $V_m$.
Then
\begin{equation}\label{tildep1m}
\widetilde {p_1^m}=\sum_{j=1}^{\nu_m}a^{(m)}_j(v)b^{(m)}_j(z)\ ,
\end{equation}
with the $b_j^{(m)}$ non-trivial.

Let $W_m$ denote the linear span of the $b^{(m)}_j$, $1\le j\le \nu_m$.
Then $W_m\sim V_m$ and
$$
W_m\subset \cH^m(\fz_0)\ .
$$ 

\item[\rm (iv)] If $|\al|=m$, then $\widetilde{p^\al}\ne0$ and
$$
\widetilde {p^\al}=\sum_{j=1}^{\nu_m}a^{(m)}_j(v)b_j^{(\al)}(z)\ .
$$
\item[\rm (v)]  For every $m$ and $k$, the products $\widetilde{p^\al}q^\beta$ with $|\al|=m$ and $[\al]+[\beta]=k$ form a basis of $\big(\cH^{m,m}(\fv)\otimes\cP^k(\fz_0)\big)^K$. In particular,
$$
\big(\cH^{m,m}(\fv)\otimes\cP^k(\fz_0)\big)^K=\big(V_m\otimes\cP^k(\fz_0)\big)^K\ .
$$
\item[\rm (vi)] The spaces $V_m$ are mutually $K$-inequivalent.
\end{enumerate}
\end{proposition}
\begin{proof}
(i) is a consequence of the structure of the $p_j$. If $k<m$, a monomial in the $p,q,r$ must necessarily contain some $r$-factor.

(ii) follows from the fact that $p_1^m$ is the only monomial in 
$\big(\cP^{m,m}(\fv)\otimes\cP^{m}(\fz_0)\big)^K$ which does not contain $r$-factors.
If we had $\widetilde{p^\al}=0$, this would establish an algebraic relation among the fundamental invariants, in contrast with \cite[Corollary 7.6]{FRY}. This last remark also proves (v) and the first statement in (iv).

The proof of (iii) requires some discussion of $\cP^m(\ell)$. 
First of all, every element of $\big(\cH^{m,m}(\fv)\otimes\cP^k(\fz_0)\big)^K$ necessarily belongs to the smaller space $\big(\big(\cH^{m,m}(\fv)\cap\cP^m(\ell)\big)\otimes\cP^k(\fz_0)\big)^K$, by the structure of the invariants.

The second fact is that the equivariant map of Remark \ref{tablecomments} (e), from $\fz_0$ to the span of the $\ell_j$, induces a surjective equivariant map from $\cP^m(\fz_0)$ to $\cP^m(\ell)$. 

Consider now $\cH^{m,m}(\fv)\cap \cP^m(\ell)$. For every irreducible $K$-invariant subspace $V$ of it, there must be an equivalent irreducible subspace $W$ in $\cP^m(\fz_0)$. This gives rise to an invariant $p_{V,W}$ of \eqref{dualbases}, belonging to 
$(V\otimes W)^K\subset\big(\cH^{m,m}(\fv)\otimes\cP^m(\fz_0)\big)^K$. But by (ii), this space is one-dimensional. Therefore there exists a unique $V\subset \cH^{m,m}(\fv)\cap \cP^m(\ell)$ and a corresponding unique $W\subset \cP^m(\fz_0)$ equivalent to $V$. This forces $V$ to be all of $\cH^{m,m}(\fv)\cap \cP^m(\ell)$, and it must coincide with $V_m$. 
The equality $\dim(V\otimes V)^K=1$ implies also that $V$ is absolutely 
irreducible. 

Decompose now $\widetilde{p_1^m}$ as
$$
\widetilde{p_1^m}=p^\sharp(v,z)+\sum_{[\beta]>0}q^\beta(z)p^\flat_\beta(v,z)\ ,
$$
with $p^\sharp\in V_m\otimes\cH^m(\fz_0)$ and $p^\flat_\beta\in V_m\otimes\cH^{m-[\beta]}(\fz_0)$. Then $p^\sharp$ and the $p^\flat_\beta$'s are all $K$-invariant. It follows from (i) that $p^\flat_\beta=0$ for every $\beta$, i.e., $\widetilde{p_1^m}=p^\sharp\in V_m\otimes\cH^m(\fz_0)$.

To complete the proof of (iv), take any element $p$ of $\big(\cH^{m,m}(\fv)\otimes\cP^k(\fz_0)\big)^K$. By \eqref{irreducibles}, 
$$
p=\sum_j p_{V_j,W_j}\ ,
$$
with the $p_{V_j,W_j}$ as in \eqref{dualbases}. Repeating the same argument used above, each $V_j$ gives rise to an invariant polynomial in $\big(\cH^{m,m}(\fv)\otimes\cP^m(\fz_0)\big)^K$. By (iii), $V_j=\cH^{m,m}(\fv)\cap\cP^m(\ell)$ for every $j$.

We prove (vi) by contradiction. If we had $V_m\sim V_{m'}$ with $m< m'$, the polynomial $\sum_ja_j^{(m')}(v)b_j^{(m)}(z)$ would be a non-zero element of $\big(\cH^{m',m'}(\fv)\otimes\cP^m(\fz_0)\big)^K$, contradicting (i).
\end{proof}

Consider now the pairs of lines 4, 5. Introducing bi-degrees for polynomials on $\fz_0$, we obtain the following rather obvious variants, on the basis of Remark \ref{tablecomments} (b).

\begin{proposition}\label{structure'}
Let $K$, $\fv$, $\fz_0$ be as in Table~\ref{invariants}, lines 4, 5.
\begin{enumerate}
\item[\rm (i')] If $k<m$, $\big(\cH^{m,m}(\fv)\otimes\cP^{k,k}(\fz_0)\big)^K$ is trivial.
\item[\rm (ii')]  The polynomials $p^\al$, $\widetilde{p^\al}$ coincide with those of line 6, evaluated at $(v,-i\bar zz)$. In particular, {\rm (ii)}, {\rm (iii)}, {\rm (iv)}, {\rm (v)}, {\rm (vi)} of Proposition \ref{structure} have the same formulation (up to the obvious notational changes), with the same $a_j^{(m)}$ and $V_m$ as for the twin pair of line 6.
\item[\rm (iii')] 
For $k\ge m$, $\big(\cH^{m,m}(\fv)\otimes\cP^{k,k}(\fz_0)\big)^K=\span\{\widetilde{p^\al} q^\beta:|\al|=m\,,\,[\![\al]\!]+[\![\beta]\!]=(k,k)\}$. In particular,
$$
\big(\cH^{m,m}(\fv)\otimes\cP^{k,k}(\fz_0)\big)^K=\big(V_m\otimes\cP^{k,k}(\fz_0)\big)^K\ .
$$
\item[\rm (iv')] If $k_1\ne k_2$, $\big(\cH^{m,m}(\fv)\otimes\cP^{k_1,k_2}(\fz_0)\big)^K$ is trivial, except at line 5, for $k_1-k_2=jn$, $j\in\bZ$. In this case,
$$
\big(\cH^{m,m}(\fv)\otimes\cP^{k_1,k_2}(\fz_0)\big)^K=\begin{cases}({\rm Pf}\,z)^j\big(\cH^{m,m}(\fv)\otimes\cP^{k_2,k_2}(\fz_0)\big)^K&\text{ if }j>0\ ,\\
(\overline{{\rm Pf}\,z})^{-j}\big(\cH^{m,m}(\fv)\otimes\cP^{k_1,k_1}(\fz_0)\big)^K&\text{ if }j<0\ .
\end{cases}
$$
\end{enumerate}
\end{proposition}

Notice that Propositions \ref{structure} and  \ref{structure'} show that, 
for every $\al$,
\begin{equation}\label{W}
\widetilde{p^\al}=p_{V_m,W_\al}\ ,
\end{equation}
with $m=|\al|$ and $W_\al\subset \cP^{[\al]}(\fz_0)$ (resp. $W_\al\subset \cP^{[\![\al]\!]}(\fz_0)$) equivalent to $V_m$.

\begin{corollary}\label{basis}
The polynomials $\widetilde{p^\al}q^\beta r^\gamma$ form a basis of $\cP(\fv\oplus\fz_0)^K$.
\end{corollary}

\vskip1cm
 
 \section{Fourier analysis of $K$-equivariant functions on $\check N$}\label{sec_fourier}
 
 \bigskip

We start from a function $G$ as in Theorem \ref{main},
$$
\begin{aligned}
& G(v,\zeta,t)=\sum_{|\gamma|=k}\zeta^\gamma G_\gamma(v,t)&\qquad \text{(lines 6-10)}\ ,\\
& G(v,\zeta,t)=\sum_{|\gamma_1|+|\gamma_2|=k}\zeta^{\gamma_1}\bar \zeta^{\gamma_2} G_\gamma(v,t)&\qquad \text{(lines 4, 5)}\ ,
\end{aligned}
$$
which is $K$-invariant, and with $G_\gamma\in\cS(\check N)$ (we use the variable $\zeta$ as a reminder that, in the course of the argument, we have taken a Fourier transform in $z$).

The following statement follows from Proposition \ref{prop_cq_SMH} and Corollary \ref{basis}.  
 
\begin{lemma} \label{decomposeG}
\quad
\begin{enumerate}
\item[\rm(i)] (lines 6-10)
A function $G\in\big(\cS(\check N)\otimes\cP^k(\fz_0)\big)^K$ can be uniquely decomposed as
\begin{equation}\label{tilde-decomposition}
G(v,\zeta,t)=\sum_{[\al]+[\beta]=k}q^\beta(\zeta)\widetilde{p^\al}(v,\zeta) g_{\al\beta}(v,t)\ ,
\end{equation}
with $g_{\al\beta}\in\cS(\check N)^K$ depending continuously on $G$. 
\item[\rm(ii)] (lines 4, 5)
A function $G\in\big(\cS(\check N)\otimes\cP^{k,k}(\fz_0)\big)^K$ can be uniquely decomposed as
\begin{equation}\label{tilde'-decomposition}
G(v,\zeta,t)=\sum_{[\![\al]\!]+[\![\beta]\!]=(k,k)}q^\beta(\zeta)\widetilde{p^\al}(v,\zeta) g_{\al\beta}(v,t)\ ,
\end{equation}
with $g_{\al\beta}\in\cS(\check N)^K$ depending continuously on $G$. 
\item[\rm(iii)] For the pair at line 5, $\big(\cS(\check N)\otimes\cP^{k+jn,k}(\fz_0)\big)^K$ equals $({\rm Pf}\,z)^j\big(\cS(\check N)\otimes\cP^{k,k}(\fz_0)\big)^K$ for $j>0$, $(\overline{{\rm Pf}\,z})^{-j}\big(\cS(\check N)\otimes\cP^{k,k}(\fz_0)\big)^K$ for $j<0$.
\end{enumerate}
\end{lemma}

From the right-hand side of \eqref{tilde-decomposition}, or of \eqref{tilde'-decomposition}, we extract the single term 
$$
\widetilde{p^\al}(v,\zeta)\tilde g_{\al\beta}(v,t)=\sum_{j=1}^{\nu_m}a_j^{(m)}(v) g_{\al\beta}(v,t)b_j^{(\al)}(\zeta)\ ,
$$
with $m=|\al|\ge1$. 

In order to emphasise that the following analysis depends only on $m$ and not on the specific multi-index $\al$, it is convenient  to introduce an abstract representation space $\cV_m$ of $K$, equivalent to $V_m$, and denote by 
$\{e^{(m)}_j\}_{1\le j\le\nu_m}$ an orthonormal basis corresponding to the basis 
$\{a_j^{(m)}\}$ of $V_m$ via an intertwining operator. 

We denote by $\tau_m$ the representation of $K$ on $\cV_m$.

We regard the function $\widetilde{p^\al} g_{\al\beta}$ 
as a $\cV_m$-valued function on $\check N$:
$$
G_{\al\beta}(v,t)= g_{\al\beta}(v,t)\sum_{j=1}^{\nu_m}a_j^{(m)}(v)\,e^{(m)}_j\ .
$$

Since the $b_j^{(\al)}$ form an orthonormal basis of the space $W_\al$ in \eqref{W} and $W_\al\sim V_m\sim\cV_m$, it follows that $G_{\al\beta}$  is $K$-equivariant, i.e.,
$$
G_{\al\beta}(kv,t)=\tau_m(k)G_{\al\beta}(v,t)\ ,\qquad (k\in K)\ .
$$

In fact, we have the following characterisation of $K$-equivariant $\cV_m$-valued smooth functions.

\begin{lemma}\label{Vm-equivariant}
Let $H$ be a $\cV_m$-valued, $K$-equivariant  Schwartz function on $\check N$. Then $H$ can be expressed as
$$
H(v,t)=h(v,t)\sum_{j=1}^{\nu_m}a_j^{(m)}(v)e^{(m)}_j\ ,
$$
with $h\in\cS(\check N)^K$, depending continuously on $H$.
\end{lemma}
\begin{proof} Reversing the argument above, from a $K$-equivariant function $H(v,t)=\sum_j H_j(v,t)e^{(m)}_j$ we can construct the $K$-invariant scalar-valued function $\tilde H(v,\zeta,t)=\sum_j H_j(v,t)b^{(m)}_j(\zeta)$
which satisfies the hypotheses of  Proposition \ref{prop_cq_SMH}.
Hence $\tilde H$ can be expressed as
$$
\tilde H(v,\zeta,t)=\sum_{m'\le m}\Big(\sum_{\begin{subarray}{c}
      [\al]+[\beta]=m \\  |\al|=m'
      \end{subarray}}q^\beta(\zeta)\widetilde{p^\al}(v,\zeta) h_{\al\beta}(v,t)\Big)\ ,
$$
with each $h_{\al\beta} \in \cS(\check N)^K$. Each term in parenthesis can be turned into a $K$-equivariant function with values in $\cV_{m'}$. Since the $\cV_{m'}$ are mutually inequivalent, the only non-zero term is the one with $m'=m$.
\end{proof}

\medskip

\begin{rmk} From this point on, we may completely disregard the special cases of lines 4 and 5, because in this abstract setting they are completely absorbed by those of line 6.
\end{rmk}

\medskip

We denote by $A^{(m)}_j\in\bD(\check N)$ the differential operators obtained from the polynomials $a^{(m)}_j$ by symmetrisation. Then 
\begin{equation}\label{Mm}
M_m=\sum_{j=1}^{\nu_m}e^{(m)}_jA^{(m)}_j
\end{equation}
is a $K$-equivariant differential operator mapping scalar valued functions on $\check N$ to $\cV_m$-valued functions.

The following statement is the key step in the proof of Theorem \ref{main}.

\begin{proposition}\label{Vm-operator}
Let $G$ be a $\cV_m$-valued, $K$-equivariant Schwartz function on $\check N$. Then $G$ can be expressed as
\begin{equation}\label{Gtou}
G(v,t)=M_m h(v,t)\ ,
\end{equation}
with $h\in\cS(\check N)^K$.

More precisely, given a Schwartz norm $\|\ \|_{(p)}$, the function $h$ can be found so that, for some $q=q(m,p)$, $\|h\|_{(p)}\le C_{m,p}\|G\|_{(q)}$.
\end{proposition}

The proof requires some representation theoretic considerations that will be developed in the next subsections.

\bigskip

\subsection{The Bargmann representations of $\check N$}\hfill\\

\bigskip

The proof requires Fourier analysis on $\check N$.
As we mentioned already, $\check N$ is either a Heisenberg group or (line 7) its quaternionic analogue, with a 3-dimensional centre. It will suffice to restrict attention to the infinite-dimensional representations.
\medskip

When $\check N$ is a Heisenberg group, i.e., $\check\fn=\fv\oplus\bR$, we see from 
Table~\ref{vinberg} that $\fv$ is a complex space (whose dimension we denote by $\kappa$), with $K$ acting on it by unitary transformations. We use the Bargmann-Fock model of its representations, that we briefly describe.

If $(v_1,\dots,v_\kappa)$ are linear complex coordinates on $\fv$, the $2\kappa$ left-invariant vector fields
\begin{equation}\label{vectorfields}
Z_j=\de_{v_j}-\frac i4\bar v_j\de_t\ ,\qquad \bar Z_j=\de_{\bar v_j}+\frac i4 v_j\de_t
\ ,\qquad j=1,\ldots,\kappa,
\end{equation}
generate ${\check\fn}^\bC$.

For $\la>0$, the Bargmann representation $\pi_\la$ acts on the Fock space $\cF_\la(\fv)$, defined as the space of holomorphic functions $\ph$ on $\fv$ such that
$$
\|\ph\|_{\cF_\la}^2=(\la/2\pi)^\kappa\int_{\fv}|\ph(v)|^2e^{-\frac\la2|v|^2}\,dv<\infty\ ,
$$
and is such that
\begin{equation}\label{Z>0}
d\pi_\la (Z_j)=\de_{v_j}\ ,\qquad d\pi_\la(\bar Z_j)=-\frac\la2 v_j\ .
\end{equation}

For $\la<0$, $\pi_\la$ acts on $\cF_{|\la|}$ as $\pi_\la(v,t)=\pi_{|\la|}(\bar v,-t)$, so that the r\^oles of $Z_j$ and $\bar Z_j$ are interchanged:
\begin{equation}\label{Z<0}
d\pi_\la (Z_j)=\frac\la2 v_j\ ,\qquad d\pi_\la(\bar Z_j)=\de_{v_j}\ .
\end{equation}

By the Stone-von Neumann theorem, the Bargmann representations $\pi_\la$, $\la\ne0$, cover the whole dual object $\widehat{\check N}$ up to a set of Plancherel measure zero.
\medskip

The case $\check\fn=\fv\oplus\IM\bH$, with $\fv=\bH^n$, requires some modifications. For every $\mu\ne0$ in $\IM\bH$, with polar decomposition $\mu=\la \varsigma$, $\la=|\mu|>0$, there is an analogous representations $\pi_\mu=\pi_{\la,\varsigma}$ which factors to the quotient algebra $\check\fn_\varsigma=\fv_\varsigma\oplus(\IM\bH/\varsigma^\perp)$. This is a Heisenberg algebra, with $\fv_\varsigma$ denoting $\fv$ endowed with the complex structure induced by the unit quaternion $\varsigma$. Then $\pi_{\la,\varsigma}$ is  the Bargmann representation of index $\la$ of $\check\fn_\varsigma$, acting on the Fock space $\cF(\fv_\varsigma)$. Again, the $\pi_\mu$ cover $\widehat{\check N}$ up to a set of Plancherel measure zero.
 
 \medskip
 
For the sake of a unified discussion, we drop the subscripts $\la$ or $\mu$, and simply write $\pi$ and $\cF$. Only when strictly necessary, we will reintroduce a parameter 
$\la>0$, leaving to the reader the obvious modifications for the other cases.
  
  In all cases, the fact that $K$ acts trivially on $\check\fz$ implies that each representation as  above is stabilised by $K$. 
 In fact, if $\sigma$ denotes the representation of ${\rm U}_\kappa$ on functions on $\fv$ given by
\begin{equation}\label{sigma}
 \big(\sigma(k)\ph\big)(v)=\ph(k\inv v)\ ,
\end{equation}
one has the identity
 $$
 \pi(kv,t)=\sigma(k)\pi(v,t)\sigma(k\inv)\ .
 $$

 The representation $\pi$ maps functions $H\in\cS(\check N)\otimes \cV_m$ into operators $\pi(H)\in\cL(\cF)\otimes\cV_m\cong \cL(\cF,\cF\otimes\cV_m)$, depending linearly on $H$ and such that
 $$
 \pi(h\otimes w)=\pi(h)\otimes w\ ,\qquad h\in\cS(\check N)\ .
 $$
 
 If $H$ is $K$-equivariant, then
 \begin{equation}\label{equivarianceH}
 \pi(H)\sigma(k)=\big(\sigma\otimes \tau_m\big)(k)\pi(H)\ ,
 \end{equation}
 for all $k\in K$. Similarly, the equivariance of $M_m$ implies that, for $k\in K$,
 \begin{equation}\label{equivarianceM}
  d\pi(M_m)\sigma(k)=\big(\sigma\otimes\tau_m\big)(k)d\pi(M_m)\ ,
   \end{equation}
   i.e., $\pi(H)$ and $d\pi(M_m)$ intertwine $\sigma$ with $\sigma\otimes\tau_m$. 
   
   With an abuse of notation, we denote the restriction of $\sigma$ to $K$ by the same symbol.
   
\medskip

Since $(\check N,K)$ is a n.G.p., the representation $\sigma$ decomposes into irreducibles without multiplicities. We can  write 
\begin{equation}\label{Fdecomposition}
\cF={\sum_{\mu\in\mathfrak X }}^{\!\oplus}V(\mu)\ ,
\end{equation}
for some set $\mathfrak X$ of dominant weights $\mu$ of $K$. For each $\mu$, we denote by $R(\mu)$ the representation of $K$ with highest weight $\mu$.
Each $V(\mu)$ is contained in some $\cP^{s,0}(\fv)$ with $s=s(\mu)$, since these subspaces are obviously invariant under $\sigma$. 
%%% We call $s(\mu)$ this value of $s$.

 In particular, $V(\mu)$ consists of $C^\infty$-vectors for $\pi$, so that $d\pi(M_m)$ is well defined on $V(\mu)$.
 
 Notice that, for the pairs in the first block of Table \ref{invariants}, each $\cP^{s,0}(\fv)$ is itself irreducible. Only for the pairs in the second block, different $V(\mu)$'s may be contained in the same $\cP^{s,0}(\fv)$.
 \medskip
 
The following lemma   in invariant theory  will be important in the next proof.
 
 \begin{lemma}\label{multiplicities}
 Let $R(\mu_1),R(\mu_2),R(\mu_3)$ be three irreducible finite dimensional representations of a complex group $G$ on spaces $V_1,V_2,V_3$ respectively. Denote by $c_{\mu_i}(\mu_j,\mu_k)$ the multiplicity of $R(\mu_i)$ in 
 $R(\mu_j)\otimes R(\mu_k)$. Then
 $$
 c_{\mu_i}(\mu_j,\mu_k)=\dim(V_i'\otimes V_j\otimes V_k)^G=c_{\mu_j'}(\mu_k,\mu_i')\ ,
 $$
where $\mu'$ stands for the highest weight of the dual representation and 
$V'$ for the dual vector space of $V$.
Over $\mathbb R$ the statement modifies as follows:
 $$
\dim(V_i\otimes V_i')^G c_{\mu_i}(\mu_j,\mu_k)=
     \dim(V_i'\otimes V_j\otimes V_k)^G=
  c_{\mu_j'}(\mu_k,\mu_i')\dim(V_j\otimes V_j')^G\ .
 $$
%
%
%% the multiplicities may not be equal, but
%% $c_{\mu_i}(\mu_j,\mu_k)=0$ implies $c_{\mu_j'}(\mu_k,\mu_i')=0$ and 
%% vise versa.  
 \end{lemma}
\begin{proof}
Recall that by a straightforward consequence of Schur's lemma, 
for irreducible complex representations, we have 
$\dim(V_i'{\otimes} V_j)^G=0$ if $V_i\not\cong V_j$
and $\dim(V_i'{\otimes} V_j)^G=1$ if $V_i\cong V_j$. 
Next, $c_{\mu_i}(\mu_j,\mu_k)$ counts how many times $V_i$ appears 
in $V_j{\otimes}V_k$ and $c_{\mu'}(\mu_k,\mu_i')$ how many times 
$V_j'$ appears in $V_i'{\otimes}V_k$. 
Finally, over the real numbers the dimension of $(V_i'{\otimes} V_J)^G$
can be larger than one. 
\end{proof}

\begin{proposition}\label{intertwiningPhi}
Let $\Phi$ be a linear operator, defined on the algebraic sum of the $V(\mu)$, 
$\mu\in\mathfrak X$, with values in $\cF\otimes\cV_m$, and intertwining $\sigma$ with 
$\sigma\otimes\tau_m$. Then
\begin{enumerate}
\item[\rm(i)] for every $\mu$,
$$
\Phi:V(\mu)\longrightarrow V(\mu)\otimes\cV_m\ {\it ;}
$$
\item[\rm(ii)] $\Phi_{|_{V(\mu)}}=0$, unless $R(\mu)\subset R(\mu)\otimes\tau_m$;
\item[\rm(iii)] $\Phi_{|_{V(\mu)}}=0$ if $s(\mu)<m$.
\end{enumerate}
\end{proposition}
\begin{proof} 
Let $P_\mu$ be the orthogonal projection of $\cF$ onto $V(\mu)$. If $\mu_1\in\fX$, 
$(P_{\mu_1}\otimes {\rm Id})\Phi_{|_{V(\mu)}}$ intertwines $R(\mu)$  with 
$R(\mu_1)\otimes\tau_m$. Hence 
$(P_{\mu_1}\otimes {\rm Id})\Phi_{|_{V(\mu)}}=0$ unless 
$R(\mu)\subset R(\mu_1)\otimes\tau_m$. 

Take $W_m$, the linear span of the polynomials $b^{(m)}_j(z)$ in \eqref{tildep1m}, as a concrete realisation of 
$\tau_m$.
Take also $V(\mu_1)$ as a concrete realisation of $R(\mu_1)$
 and $\overline{V(\mu)}$  as concrete realisation of the (complex) contragredient representation $R(\mu)'$ of $R(\mu)$. 
By Lemma~\ref{multiplicities},
\begin{equation}\label{leftright}
R(\mu)\subset R(\mu_1)\otimes\tau_m \Longleftrightarrow 
 \big(\overline{V(\mu)}{\otimes} V(\mu_1){\otimes} W_m\big)^K\ne\{0\}\ .
\end{equation}

By Remark~\ref{stabiliser},
for a nonzero element $\zeta\in\fz_0$,  
the pair $(\check N,K_\zeta)$ is also a nilpotent Gelfand pair, so that $\cF(\fv)$ decomposes without multiplicities under the action of $K_\zeta$. 
Let $p(v,z)$ be a nonzero element of 
$\big(\overline{V(\mu)}{\otimes} V(\mu_1){\otimes} W_m\big)^K$, and fix $\zeta\in \fz_0$ such that $p_0(v)=p(v,\zeta)$ is not identically zero. Then $p_0$ is 
$K_\zeta$-invariant and contained in $\overline{V(\mu)}{\otimes}{V(\mu_1)}$. Hence $V(\mu)$ and $V(\mu_1)$ must contain two $K_\zeta$-invariant, irreducible, equivalent subspaces. By multiplicity freeness, this forces that $\mu=\mu_1$ and we obtain (i).

At this point, (ii) is obvious.

To verify (iii), observe that %%% any nonzero element of 
the subspaces $V_m$ are mutually inequivalent by 
Propositions~\ref{structure}(vi), 
\ref{structure'}(ii'). Hence 
$V_m$ does not appear in $\cP^{s,s}(\fv)$ for 
$s<m$.
%% if $s(\mu)<m$ by the same reason as in the proof of %%
%% must be divisible by $\widetilde{p_1^m}$, and this cannot happen 
%% unless $s(\mu)\ge m$.
\end{proof}

\bigskip

\subsection{Multiplicity of $R(\mu)$ in $R(\mu)\otimes\cV_m$}\hfill

\bigskip

We need at this point to obtain, for any $m$,
\begin{enumerate}
\item [(a)] a precise description of the ``$m$-admissible'' weights $\mu$, i.e., such that $R(\mu)\subset R(\mu)\otimes\cV_m$;
\item [(b)] that, for such a pair, $R(\mu)$ is contained in $R(\mu)\otimes\cV_m$ without multiplicities.
\end{enumerate}

Point (a) above forces us to go into a case by case analysis, from which we will obtain sets of parameters for the $m$-admissible weights. This analysis will also give us a positive answer to point (b). 
%% Note that 
%%% $\dim(\cW_m\otimes\cW_m)^K=1$ by propositions~\ref{structure}\,,\ref{structure'}. 

For a simple complex (or compact) group, we let 
$\varpi_i$ denote its fundamental dominant weights. 

\medskip

\subsubsection{Pairs in the first block of Table \ref{invariants}}\quad

\medskip

In these cases we know that $V(\mu)=\cP^{s,0}(\fv)$ for some $s$.

\begin{proposition}\label{firstblock}
Let $\fv=\bC^n$, with $K={\rm (S)U}_n$, or $\fv=\bC^{2n}$ with $K={\rm Sp}_n$. Then $\cP^{s,0}(\fv)$ is contained in $\cP^{s,0}(\fv)\otimes \cV_m$ if and only if $s\ge m$, and in this case with multiplicity one.
\end{proposition}
\begin{proof} 
We know from Propositions~\ref{structure} (i) and \ref{structure'} (i'), that $\cP^{s,0}(\fv)$ is not contained in $\cP^{s,0}(\fv)\otimes \cV_m$ if $s<m$. We suppose now that $s\ge m$ and apply the equivalence \eqref{leftright}. Since the only fundamental invariant depending only on $v$ is $|v|^2$, there is exactly one invariant (up to scalars) in $\cP^{s,0}(\fv)\otimes \cP^{0,s}(\fv)\otimes W_m$, namely $|v|^{2(s-m)}\widetilde{p_1^m}$.

By Lemma \ref{multiplicities}, this gives existence and uniqueness of a subspace of 
$\cP^{s,0}(\fv)\otimes \cV_m$ equivalent to $\cP^{s,0}(\fv)$.
\end{proof}
\medskip

\medskip

\subsubsection{Pair of line 8}\quad

\medskip

%% Denote by $\varpi_1$ (resp. $\varpi_2$) the highest weight of the 
%% contragredient of the defining representation of ${\rm SU}_n$ 
%% (resp. of its extension to $\Lambda^2\bC^n$). 

The action of ${\rm SU}_n$ on the ${\mathbb C}^{n}$-factor in $\gt v$
extends to the action of ${\rm SL}_n(\mathbb C)$. 
Depending on the ordering of simple roots, this latter action may have 
the highest weight either $\varpi_1$ or
$\varpi_{n-1}$. For convenience we assume that this highest weight is $\varpi_{n-1}$.
Also let $S^i$ denote the representation of ${\rm SU}_2$ on $\cP^{i,0}(\bC^2)$ and by $\chi^s$ the $s$-th power of the identity character on ${\rm U}_1$.
Then, cf. \cite{knop-mf},
$$
\sigma_{|_{\cP^{s,0}(\fv)}}=\sum_{i+2j=s}R(i\varpi_1+j\varpi_2)\otimes S^i\otimes\chi^s\ .
$$

We call $R_{s,i}$ (with $0\le i\le s$, $s-i\in2\bN$) the $i$-th summand above, and $V_{s,i}$ the corresponding subspace of $\cP^{s,0}(\fv)$.

\begin{proposition}\label{line8}
$R_{s,i}$ is contained in $V_{s,i}\otimes \cV_m$ if and only if $i\ge m$, and in this case with multiplicity one.
\end{proposition}
\begin{proof}
Notice that both ${\rm SU}_n$ and the centre of ${\rm U}_2$ act trivially on $\fz_0$ and that the remaining factor ${\rm SU}_2$ of $K$ acts on $W_m$ by $S^{2m}$. Then we want to find when it is true that $R_{s,i}\subset R_{s,i}\otimes S^{2m}$. We have
\begin{equation}\label{containment}
\begin{aligned}
R_{s,i}\otimes S^{2m}&=R(i\varpi_1+j\varpi_2)\otimes (S^i\otimes S^{2m})\otimes\chi^s\\
&=R(i\varpi_1+j\varpi_2)\otimes (S^{i+2m}\oplus S^{i+2m-2}\oplus\cdots\oplus  S^{|2m-i|})\otimes\chi^s\ .
\end{aligned}
\end{equation}

It is quite clear that we find the summand $S^i$ in the sum in parentheses if and only if $i\ge|2m-i|$, i.e., $i\ge m$, and in this case it appears once and only once.
\end{proof}

\medskip

\subsubsection{Pair of line 9}\label{subspaces}
\quad

\medskip

With the same notation of the previous case, we have, cf. \cite{knop-mf},
$$
\sigma_{|_{\cP^{s,0}(\fv)}}=\sum_{\begin{subarray}{c}i+2j\le s\\  s-i\in2\bN\end{subarray}}R(i\varpi_1+j\varpi_2)\otimes S^i\otimes\chi^s=\sum_{\begin{subarray}{c}i+2j\le s\\  s-i\in2\bN\end{subarray}}R_{s,i,j}\ .
$$

\begin{proposition}\label{line9}
$R_{s,i,j}$ is contained in $V_{s,i,j}\otimes \cV_m$ if and only if $i\ge m$, and in this case with multiplicity one.
\end{proposition}
\begin{proof}
As before, we want to find when it is true that $R_{s,i,j}\subset R_{s,i,j}\otimes S^{2m}$. The same identity \eqref{containment} as above holds and we obtain the same conclusion.
\end{proof}

\medskip

\subsubsection{Pair of line 10}\quad

\medskip

We can identify $\fv$ with $\bC^8$, with ${\rm Spin}_7$ acting via the spin representation and ${\rm U_1}$ by scalar multiplication. 

The spin representation defines an embedding of 
${\rm Spin}_7$ into ${\rm SO}_8$.
Under the action of ${\rm U}_1\times{\rm SO}_8$, $\cP^{s,0}(\bC^8)$ decomposes into irreducibles as
$$
\cP^{s,0}(\bC^8)=\sum_{\begin{subarray}{c}i\ge0\\  s-i\in2\bN\end{subarray}}n(v)^{s-i}\cH^i=\sum_{\begin{subarray}{c}i\ge0\\  s-i\in2\bN\end{subarray}}V_{s,i}\ ,\qquad n(v)^2=v_1^2+\cdots+v_8^2\ ,
$$
see e.g. \cite[\S19.5]{FH-Rep}. 

The compact groups ${\rm Spin}_7$ and ${\rm SO}_8$ have the same invariants on 
$\bC^8$, see e.g. \cite[Theorem~7.5(8)]{FRY}.
Following the same line of arguments as was used in Section~\ref{sec_polynomials}, 
we can conclude that 
the above decomposition is also irreducible under the action of ${\rm Spin}_7$. 

%%% Denoting by $R(\varpi_1)$ the defining representation of 
%% ${\rm SO}_7={\rm Spin}_7/\{\pm I\}$ and by $R(\varpi_3)$ the spin representation,
Therefore
$$
\sigma_{|_{\cP^{s,0}(\fv)}}=\sum_{\begin{subarray}{c}i\ge0\\  s-i\in2\bN\end{subarray}}R\big(i\varpi_3\big)\otimes\chi^s=\sum_{2i\le s}R_{s,i}\ .
$$

\begin{proposition}\label{line10}
$R_{s,i}$ is contained in $V_{s,i}\otimes \cV_m$ if and only if $i\ge m$, and in this case with multiplicity one.
\end{proposition}
\begin{proof}
The group ${\rm Spin}_7$ acts on $\fz_0$ via $R(\varpi_1)$ (and ${\rm U}_1$ acts trivially).
%%% we have $W_m\sim W'_m\sim\cW_m$. 
The orthogonal projection of $W_m$ on the highest  component $R(m\varpi_1)$ of $\cP^m(\fz_0)$ must be non-zero, otherwise 
$\cV_m\subset\cP^{m-2}(\fz_0)$ and
we would have an invariant contradicting  Proposition~\ref{structure}(i). 
Therefore, ${\rm Spin}_7$ acts on $\cV_m$ via $R(m\varpi_1)$.

We follow  \cite[Example~5.2]{Litt}: setting $k=m-s+i$, $R(m\varpi_1)\otimes R(k\varpi_3)$ decomposes as a direct sum
$$
R(m\varpi_1)\otimes R(k\varpi_3)=\sum_{a_1,a_2,a_3,a_4}R\big(a_1\varpi_1+a_2\varpi_2+(a_3+a_4)\varpi_3\big)\ ,
$$
extended over the quadruples $(a_j)_{1\le j\le4}$ of nonnegative integers such that
\begin{equation}\label{system}
a_1(1,0)+a_2(1,2)+a_3(0,1)+a_4(1,1)=(m,k)\ .
\end{equation}

We are interested in the solutions of \eqref{system} which satisfy the requirement $a_1=a_2=0$ and $a_3+a_4=k$. It is clear that there is one (and only one) solution if and only if $m\le k$, with $a_3=m$, $a_4=k-m$.
\end{proof}

\bigskip

\subsection{Nonvanishing of $d\pi(M_m)$ on $m$-admissible weight spaces}\hfill

\bigskip

We have shown that, if $\mu$ is $m$-admissible, there is a unique subspace $X(\mu,m)\subset V(\mu)\otimes\cV_m$ equivalent to $V(\mu)$. Therefore, Proposition \ref{intertwiningPhi} (i) can be made more precise by saying that an operator $\Phi$ intertwining $\sigma$ with $\sigma\otimes\tau_m$ maps
$V(\mu)$ into $X(\mu,m)$ for any $m$-admissible $\mu$. Moreover, $\Phi_{|_{V(\mu)}}$ is uniquely determined up to a scalar factor.

Assume that the identity \eqref{Gtou} holds. Applying $\pi$ to both sides, we obtain
$$
\pi(G)=d\pi(M_m)\pi(h)\ .
$$

In this identity, $\pi(G)$ and $d\pi(M_m)$ satisfy the assumptions of Proposition \ref{intertwiningPhi}, whereas $\pi(h)$ maps each $V(\mu)$ into itself by scalar multiplication (this is the special case $m=0$ of Proposition \ref{intertwiningPhi}).

The next proposition, whose proof is postponed to the end of this section, provides a necessary condition for being able to solve equation \eqref{Gtou} in $h$.

\begin{proposition}\label{nonvanishing}
For every $m$-admissible weight $\mu$,
$d\pi(M_m)_{|_{V(\mu)}}\ne0$.
\end{proposition}

Let $C=(c_{jk})$ be a $\kappa\times\kappa$ hermitian matrix (with $\kappa=\dim_\bC \fv$), and set
$$
\ell_C(v):=\sum_{j,k}c_{jk}v_j\bar v_k\ .
$$
%%% the associated quadratic form. 
The symmetrisation process transforms $\ell_C$ into the operator $L_C\in\bD(\check N)$,
$$
L_C=\half\sum_{i,k}c_{ik}(Z_j\bar Z_k+\bar Z_kZ_j)\ ,
$$
where the $Z_j,\bar Z_j$ are the vector fields in \eqref{vectorfields}.

The image of $L_C$ in the Bargmann representations can be described in terms of the representation $\sigma$ in~\eqref{sigma}.

\begin{lemma}\label{metaplectic}
Let $C=(c_{ik})$ be a $\kappa\times\kappa$ hermitian matrix (so that $iC\in\fu_\kappa$), and let
$$
L_C=\half\sum_{i,k}c_{ik}(Z_j\bar Z_k+\bar Z_kZ_j)\in\bD(\check N)\ .
$$

Then, for $\la>0$,
$$
d\pi_\la(iL_C)=\frac\la2d\sigma(iC)\ .
$$

This identity extends by $\bC$-linearity to $C\in\fs\fl_\kappa$, understanding $L_C$ as $\half L_{C+C^*}-\frac i2 L_{i(C-C^*)}$.
\end{lemma}

For the proof, that we skip, it suffices to verify the identity for $C=E_{ik}+E_{ki}$ and $C=iE_{ik}-iE_{ki}$. Notice that $\sigma$ is the restriction to ${\rm U}_\kappa$ of the metaplectic representation.
\medskip

Denote by $L_j$ the symmetrisation on $\check N$ of the polynomials $\ell_j(v)$ appearing in the expression \eqref{pk} of the mixed invariants $p_k$.
We want to identify how $d\pi\big(\span\{L_j\}\big)$ sits inside $d\sigma(\fu_\kappa)$ and understand the action  on $V(\mu)$ of the complex algebra generated by the $d\pi(L_j)$. By Lemma~\ref{metaplectic}, this is equivalent to
identifying 
$$
\fc=\big\{iC:L_C\in\span\{L_j\}\big\}
$$ 
inside $ \fu_\kappa$ and
 study the algebra generated by $d\sigma(\fc^\bC)$.

\begin{proposition}\label{Lj-u}
 As a representation space of $K$, $\fc\sim\cV_1\sim\fz_0$. Moreover,
\begin{enumerate}
\item[\rm (i)] When $\fz_0=\su_r$ (line 6 with $r=n$, or lines 8, 9 with $r=2$), $K$ contains a factor $K_0\cong{\rm SU}_r$ acting nontrivially on $\fz_0$. Then $\fc=\fk_0$.
\item[\rm (ii)] For line 7, $\fc$ is the ${\rm Sp}_n$-invariant complement of $\fs\fp_{2n}$ in $\su_{2n}$.
\item [\rm(iii)] For line 10, let $\iota$ be the inclusion of ${\rm Spin}_7$ in ${\rm SO}_8$ given by the spin representation $R(\varpi_3)$. Then $\fc$ is the 7-dimensional ${\rm Spin}_7$-invariant complement of $d\iota(\fs\fo_7)$ in $\fs\fo_8$. 
\end{enumerate}
\end{proposition}
\begin{proof}
The first statement follows from the equivalence $\span\{L_j\}\sim\span\{\ell_j\}\sim\cV_1$.

After Lemma \ref{metaplectic}, (i) is almost tautological: the symmetrisation of 
$p_1(\cdot,z)$ is $L_{-id\sigma(z)}$.
For (ii), it is basically the same argument.

For (iii), we must recall from \cite{FRY} that the terms $v_1,v_2$ in the expression of $p_1(v,z)=\RE\big(z(v_1\bar v_2)\big)$ are octonions representing the two  components of $v=1\otimes v_1+i\otimes v_2$ in the decomposition of $\bC\otimes\bO$ as the direct sum of  $\bR\otimes\bO$ and $(i\bR)\otimes\bO$.

For fixed $z$, $p_1(\cdot,z)$ is a quadratic form  satisfying $p_1(\bar v,z)=-p_1(v,z)$ (here $\bar v=1\otimes v_1-i\otimes v_2$). In complex coordinates, it is then expressed by a hermitian matrix $C_z$ with purely imaginary coefficients. It follows that $iC_z\in\fs\fo_8$, and these elements span a ${\rm Spin}_7$-invariant 7-dimensional subspace. This is necessarily the complement of $d\iota(\fs\fo_7)$.
\end{proof}

Notice that either $\fc\subset\fk$ is already a Lie algebra, or 
$\fk\oplus\fc\subset\fu_\kappa$ is itself a Lie algebra. 
Set $\gt g:=\gt k+\fc$.
In two case, lines~7 and 9, $\gt g\ne\gt k$, %% (ii) and (iii) above, with 
when
$\fg$ is either $\fs\fu_{2n}$ or $\fg=\fs\fo_8\oplus\bR$, respectively.
%%% We call $G$ the group 
%% ${\rm SU}_{2n}$, resp. ${\rm SO}_8\times{\rm U}_1$.
Let $G$ be the corresponding compact group with $\gt g=\Lie G$. 
Also notice that if $\gt g\ne\gt k$, then, up to the summand $\bR$, 
$\fk\oplus\fc$ is the Cartan decomposition of the symmetric pair $(\fg,\fk)$.

\begin{lemma}\label{l_j-action}
%%% Assume that  $\fc$ is the complement of $\fk$ in $\fg$. Then 
The subspaces $V(\mu)$ in \eqref{Fdecomposition} are also $G$-invariant.
\end{lemma}
\begin{proof}
The action of $K$ on $\gt c$ is equivalent to the action of $K$ on 
$\gt z_0$. Therefore for each $iC\in\gt c$, the action of the stabiliser 
$K_{iC}$ on $\cF$ is multiplicity free and $iC$ preserves each of the irreducible
summands. Since $K_{iC}\subset K$, 
the action of $iC$ also preserves $K$-invariant irreducible subspaces 
in $\cF$.
 \end{proof}

\noindent
The statement of Lemma~\ref{l_j-action} can also be verified directly using the fact that 
$K$ and $G$ have the same invariants on $\fv$.
%%%, a property that can be verified directly in each of the two cases.

We can now prove Proposition \ref{nonvanishing}.

\begin{proof}[Proof of Proposition \ref{nonvanishing}]
First of all, recall that we do not treat lines 4 and 5, because they are completely covered 
by line 6. 

Fix a complex basis $\{u_1,\dots u_{\nu_1}\}$ of $\fz_0^\bC$ with $u_1%= X
$
being a lowest weight vector (of weight, say, $-\alpha$) and let $(z_1,\dots z_{\nu_1})$ denote coordinates in this basis. 
Then 
$\alpha$ is also the highest weight of $\fc^\bC$,
$z_1^m$ is a vector of the highest weight, $m\alpha$, 
%%% (equal to $m\alpha'$, if $R(\al')=R(\al)'$) 
in $\cP^m(\fz_0)$,
and the weights $\pm m\alpha$ do not appear in lower degree polynomials on 
$\gt z_0$.
Hence $\pm m\al$ are not among the weights of  $\cP^s(\ell)$ with $s<m$.
Decomposing $p_1(v,z)$ with respect to $z_j$, one gets
$$
p_1(v,z)=\sum_{j=1}^{\nu_1}a_j(v)z_j\ ,
$$
where $a_1^m$ is a lowest weight vector in $\cP^m(\ell)$. We must have 
$a_1^m\in\cH^{m,m}(\fv)$, since otherwise, by Corollary~\ref{basis}, the weight 
$-m\alpha$ would also be contained in lower degrees in $\ell$. 
By Proposition~\ref{structure},  the $K$-invariant space generated by $a_1^m$ is $V_m$. In turn, this implies that $z_1^m$ belongs to the space $W_m$ of 
Proposition~\ref{structure}(iii).

We regard $M_1$ in \eqref{Mm} as 
$$
M_1=\sum_{j=1}^{\nu_1}A_j(v)z_j\ ,
$$
identifying $\fz_0$ with $\cV_1$. Then
\begin{equation}\label{M1m}
M_1^m=\sum_{|\beta|=m}B_\beta z^\beta\ ,
\end{equation}
where each $B_\beta$ is an $m$-fold composition of the $A_j$.

Each $B_\beta$ is the symmetrisation of a polynomial $b_\beta$ depending on $v$ and $t\in\check\fz$. The polynomial
$$
P(v,z,t)=\sum_\beta b_\beta(v,t)z^\beta
$$
is $K$-invariant, and its component of highest degree in $v$ is $p_1^m$. Therefore, $\widetilde{p_1^m}$ is the highest weight term in the decomposition \eqref{irreducibles} of $P$. 

In particular, $M_1^m$ and $M_m$ have the same highest weight component.  
Then
$$
\lan M_m,z_1^m\ran=\lan M_1^m,z_1^m\ran=A_1^m\ .
$$

Let $X$ be the lowest weight element %%% of $X\in\fc^\bC$ 
in $\fc^\bC$
such that
 $A_1=L_{X}$. 
Lemma \ref{metaplectic} implies that
$$
d\pi_\la(A_1^m)=(\la/2)^md\sigma(X)^m\ .
$$

Therefore it remains to show that under the identification $\gt z_0=\gt c$, the 
element $d\sigma(X)^m$
does not vanish on $V(\mu)$. 

As an illustration, consider first the example of line 6. 
Here $\fc=\fk_0$ and $X$ is a lowest root vector in $\mathfrak{sl}_n$. 
The complex group ${\rm SL}_n(\mathbb C)$ acts on 
$V(\mu)$ via $R(s\varpi_1)$ with $s\ge m$. 
Clearly $d\sigma(X^m)$ is non-zero on the highest weight  vector 
of $V(\mu)$. 

In general, we argue in the following way.
The action of $X$ on polynomials on $\fv$ is completely determined by 
the action of $X$ on $\fv$ itself or by the representations of the group $G$. 
%%% or $K$, according to the case). 
If $V(\mu)$ is $m$-admissible  and $d\sigma(X^m)$ is zero on $V(\mu)$, 
then it is also zero on the contragredient  space $\overline{V(\mu)}$, and, hence,
$d\sigma(X^{2m})$ vanishes on a copy of $V_m$ sitting inside 
$V(\mu){\otimes}\overline{V(\mu)}\subset \cP^{s,s}$.

Now $V_m$ has the highest weight $m\alpha$ and $X$ is of weight $-\alpha$.
Since $X$ is a weight vector (with a nonzero weight) 
of a torus in $\gt g^{\mathbb C}$, 
it is necessary a nilpotent element.
Therefore
one can include it into an $\mathfrak{sl}_2$-triple 
$\{X,H,Y\}\subset\mathfrak g^\mathbb C$, where 
the semisimple element $H$ is contained in $\gt k^{\mathbb C}$.
(If $\gt k=\gt g$ this is Jacobson-Morozov theorem, in the two cases
with $\gt g\ne\gt k$ 
the claim follows from the fact that $(\gt g,\gt k)$ is a symmetric pair, 
see \cite[Prop.~4]{KR69}.) 

Then $H$ multiplies a highest weight vector $v\in V_m$ by 
$2m$, therefore $v$ gives rise to at least one 
irreducible representation of 
$\{X,H,Y\}$ of dimension at least 
$(2m{+}1)$. By the linear algebra considerations, 
$d\sigma(X^{2m})v\ne 0$.
A more careful analysis can show that 
$d\sigma(Y)v=0$ and 
$d\sigma(X^{2m})v$
is a  lowest weight vector of $V_m$.  
\end{proof}

\vskip1cm

\section{Proof of Proposition \ref{Vm-operator}}\label{sec_N'nonab}

\bigskip

First, let us fix some notation.
Let $T=\partial_t$ be the central derivative of $\check N$ when $\check N$ is the Heisenberg group. For the pair at line 7, where $\check N$ is the quaternionic Heisenberg group, we take
$T_j=\partial_{t_j}$, $j=1,2,3$, the derivatives in three orthogonal coordinates on $\check\fz$.

We can assume that ${\check D}=(\check D_1,\ldots,\check D_{d_0-1},i\inv T)$ and ${\check D}=(\check D_1,\ldots,\check D_{d_0-3},i\inv T_1,i\inv T_2,i\inv T_3)$ respectively. The first $d_0-1$ (resp. $d_0-3$) operators come from symmetrisation of the polynomials $\rho_j\in\rho_\fv$.  We convene that $\check D_1$ is the sublaplacian, i.e., the symmetrisation of $|v|^2$.
A point of the spectrum $\Sigma_{{\check D}}$ of $(\check N,K)$ can then be written as $\xi'=(\tilde \xi,\lambda)$ with $\lambda$ in $\bR$ or $\bR^3$, depending on the pair considered.
The points of the spectrum with $\lambda\not=0$ form a dense subset of $\Sigma_{{\check D}}$ and they are parametrised by $\la$ and $\mu\in \fX$ as $\xi'(\lambda, \mu)$,
where $\xi'_j(\la,\mu)$ is the scalar such that 
\begin{equation}\label{tildexi}
d\pi_\la(\check D_j)_{|_{V(\mu)}}=\xi_j(\la,\mu){\rm Id}\ .
\end{equation}

Note that $\xi_{d_0}(\la,\mu)=\la$ and, if $V(\mu)\subset \cP^{s,0}(\fv)$, then $\xi'_1(\lambda,\mu)=|\lambda| (2s+\kappa)$, cf., e.g., \cite{ADR1}.

By $\del_j$ we denote the degree of homogeneity of the polynomial $\rho_j$ (and hence of $\check D_j$) with respect to the automorphic dilations 
\begin{equation}\label{dilation_fn'}
r\cdot (v,t)=(r^\half v,rt)
\end{equation} 
of $\check\fn$ (and of $\check N$); i.e., $\del_j=\half\deg\rho_j$ for the first $d_0-1$ (resp. $d_0-3$) operators, and $\del_j=1$ for the $T$'s.

If $\ph(v,t)$ is a spherical function, then $\ph_r(v,t)=\ph(r^\half v,rt)$ is also spherical, and 
$\xi'_j(\ph_r)=r^{\del_j}\xi'_j(\ph)$.
Then
 $\Sigma_{{\check D}}$ is invariant under 
the following dilations of $\bR^{d_0}$:
\begin{equation}
\label{dilation_Rd0}
r\cdot (\xi'_1,\ldots,\xi'_{d_0})=
(r^{\delta_1}\xi'_1,\ldots,r^{\delta_{d_0}}\xi'_{d_0})
\ .
\end{equation}

In terms of the parameters $(\la,\mu)$, we have
$$
r \cdot \xi'(\lambda, \mu)
=
\xi'(r^2\lambda,\mu)\ .
$$

Now we define the following left-invariant, self-adjoint differential operator on $\check N$:
$$
U_m=M_m^*M_m =\sum_{j=1}^{\nu_m} {A^{(m)}_j}^* A^{(m)}_j 
\ .
$$

Note that
$$
\ker U_m=\bigcap_{j=1}^{\nu_m} \ker A^{(m)}_j 
\ .
$$

As $a^{(m)}_j\in \cP^{m,m}(\fv)$, the operators $A^{(m)}_j$ and $U_m$ are homogeneous of degree $m$
and $2m$, respectively, w.r. to the dilation \eqref{dilation_fn'}. Furthermore as $M_m$ is $K$-invariant, $U_m$ is also $K$-invariant.
Hence it can be written as $U_m=u_m(\check \cD)$ where $u_m\in \cP(\bR^{d_0})$ is homogeneous of degree $2m$ 
with respect to the dilations \eqref {dilation_Rd0} of $\bR^{d_0}$.

By \eqref{tildexi},
$$
d\pi_\lambda(U_m)_{|V(\mu)}=u_m (\xi'(\lambda,\mu)){\rm Id}
\ .
$$

Let
$$
S_m= \{\xi'\in\Sigma_{{\check D}}, \ u_m(\xi')=0\}\ .
$$

Then
\begin{equation}\label{kerUm}
\ker U_m\cap \cS(\check N)^K=\{f:\supp\cG f\subset S_m\}\ .
\end{equation}

Moreover, $S_m$ is invariant under the dilations \eqref{dilation_Rd0}.

The next lemma shows that polynomials which vanish on $S_m$ can be divided by $u_m$.

\begin{lemma}
\label{lem_claim}
Assume that $p\in \cP(\bR^{d_0})$ vanishes on $S_m$. Then $p$ is divisible by $u_m$.
\end{lemma}
\begin{proof} We may assume that $p$ is homogeneous with respect to the dilations \eqref {dilation_Rd0} of $\bR^{d_0}$.

Consider first the pairs in the first block of Table~\ref{invariants}.

In this case there is only one invariant in $\rho_\fv$, leading to the sublaplacian on $\check N$, and then only one coordinate $\xi'_1$ besides those corresponding to the $T$'s.
The space $V(\mu)$ coincides with $\cP^{s,0}(\fv)$ and 
by Proposition \ref{firstblock}
$\cP^{s,0}(\fv)\subset \cP^{s,0}(\fv)\otimes \cV_m$
if and only if $s\geq m$.
By Proposition~\ref{intertwiningPhi}, 
$d\pi(M_m)$ vanishes on $\cP^{s,0}$, if $s<m$.
This is also the case for $d\pi(U_m)=d\pi(M_m)^*d\pi(M_m)$.
Hence the set $S_m$ contains 
all the points of the form $(|\lambda|(2s+\kappa),\lambda)$ for any $\lambda\in \bR^{\dim \check\fz}$ and $s=0,\ldots, m-1$.

We decompose $p$ into its odd and even part w.r. to $\xi'_1$ as
$$
p(\xi'_1,\lambda)=
\xi'_1\, p_1({\xi'_1}^2,\lambda)
+p_2({\xi'_1}^2,\lambda)
\ ,
$$
where $p_1$ and $p_2$ are two polynomials with suitable homogeneity.

We claim that $p_1$ and $p_2$ must both vanish on the set of points
$(|\lambda|^2(2s+\kappa)^2,\lambda)$ with $\lambda\in \bR^{\dim \check\fz}$ and $s=0,\ldots, m-1$. If it were not so, we would have the identity
$$
|\la|(2s+\kappa)=-\frac{p_2\big(|\la|^2(2s+\kappa)^2,\lambda\big)}{p_1\big(|\la|^2(2s+\kappa)^2,\lambda\big)}\ .
$$

 This contrasts with the fact that the right-hand side is a rational function in $\la$, while the left-hand side is not. 
Then
$p_1(\eta,\lambda)$ and $p_2(\eta,\lambda)$ are both divisible by $\prod_{s=0}^{m-1}(\eta -(2s+\kappa)^2|\lambda|^2)$.
Therefore $p(\xi'_1,\la)$ is divisible by $\prod_{s=0}^{m-1}({\xi'_1}^2 -(2s+\kappa)^2|\lambda|^2)$. This also holds for $p=u_m$. Hence
\begin{equation}\label{um-1}
u_m(\xi'_1,\la)=c\prod_{s=0}^{m-1}({\xi'_1}^2 -(2s+\kappa)^2|\lambda|^2)\ .
\end{equation}

\medskip

We consider next the pairs in the second block of Table \ref{invariants}.

There are two invariants in $\rho_\fv$ for the pairs at lines 8 and 10 and three for the pair at line 9.
In the notation of Subsection \ref{subspaces}, the space $V(\mu)$ coincides with $V_{s,i}$ or $V_{s,i,j}$ respectively, always with $i$ and $s$ of the same parity. We adopt the notation
$$
\xi'(\la,\mu)=\begin{cases}\xi'(\la,s,i)&\text{(lines 8,10)}\ ,\\ \xi'(\la,s,i,j)&\text{(line 9)}\ .\end{cases}
$$

More precisely, $\xi'_1=|\la|(2s+\kappa)$ only depends on $\la$ and $s$. For the pair at line 9, $\xi'_2$ only depends on $\la,s,i$, because it is invariant under the larger group ${\rm U}_2\times {\rm SU}_{2n}$.

The  homogeneity degrees of the elements of ${\check D}$ w.r. to the dilations \eqref{dilation_fn'} are $(1,2,1)$ at lines 8 and 10, and $(1,2,2,1)$ for the pair at line 9. By \eqref{dilation_Rd0} and the subsequent comments, 
\begin{equation}\label{scaling}
\xi'_1(\la,s)=|\la|\xi'_1(1,s)\ ,\quad \xi'_2(\la,s,i)=\la^2\xi'_2(1,s,i)\ ,\quad \xi'_3(\la,s,i,j)=\la^2\xi'_3(1,s,i,j)\ .
\end{equation}

We split $\Sigma_{{\check D}}$ as the union of $\Sigma_{{\check D}}^\flat=\{\xi':\xi'_{d_0}=0\}=\rho_\fv(\fv)\times\{0\}$, cf. \cite{ADR2}, and the sets
$$
\tilde S_i=\begin{cases} \{\xi' (\lambda,s,i)\ , \ \lambda\in \bR\ ,\ s\in i+2\bN
\}\ , &\text{(lines 8, 10)}\\
 \{\xi' (\lambda,s,i,j)\ , \ \lambda\in \bR\ ,\ s\in i+2\bN\ ,\ 0\le j\le (s-i)/2
\}\ ,&\text{(line 9)}
\end{cases}
$$
depending on $i\ge0$.

By Propositions \ref{line8}, \ref{line9}, and \ref{line10},
$R_{s,i}$ (resp. $R_{s,i,j}$) is contained in $V_{s,i}\otimes \cV_m$ 
(resp. $V_{s,i,j}\otimes \cV_m$)
if and only if $i\geq m$.
By Proposition \ref{intertwiningPhi} and Proposition \ref{nonvanishing} 
$d\pi(M_m)$ vanishes on $V(\mu)$
if and only if $R(\mu)$ is not included in $V(\mu)\otimes \cV_m$,
which means $i<m$.
This is also the case for $d\pi(U_m)=d\pi(M_m)^*d\pi(M_m)$.

Hence  $S_m$ contains the union of sets $\tilde S_i$
for $0\le i\le m-1$.
Moreover, each polynomial $u_m$ vanishes on $\tilde S_i$, $i<m$, but is never zero on $\tilde S_i$, $i\geq m$, except for the origin.

We prove recursively the existence of polynomials $\tilde u_i\in \cP(\bR^{d_0})$, $i\ge0$,
such that
\begin{itemize}
\item[(a)]
$\tilde u_i(\xi'_1,\xi'_2,\lambda)= c_{1,i}{\xi'_1}^2+\xi'_2+ d_i\lambda^2$, resp.
(for line 9),
$\tilde u_i(\xi'_1,\xi'_2,\xi'_3,\lambda)= c_{1,i}{\xi'_1}^2+\xi'_2+ c_{3,i} \xi'_3+ d_i\lambda^2$;
\item[(b)]
 each $\tilde u_i$ vanishes on $\tilde S_i$ 
but does not vanish  on any other $\tilde S_{i'}$, $i'\not= i$, except for the origin;
\item[(c)] $u_m$ is a scalar multiple of $\prod_{i=0}^{m-1}\tilde u_i$.
\end{itemize}

Once this is done, the proof can be concluded as in the previous case.

Consider the polynomial $u_1$. 
Being homogeneous of degree 2, it must be of the form
\begin{equation}\label{u1-810}
u_1(\xi'_1,\xi'_2,\lambda)= a_1{\xi'_1}^2+a_2\xi'_2+b\lambda^2+c\xi'_1\la\ ,
\end{equation}
resp.
\begin{equation}\label{u1-9}
u_1(\xi'_1,\xi'_2,\xi'_3,\lambda)=a_1{\xi'_1}^2+a_2\xi'_2+a_3\xi'_3+b\lambda^2+c\xi'_1\la
\ .
\end{equation}

For $i=0$, we have
$$
a_1\la^2{\xi'_1}^2(1,s)+a_2\la^2\xi'_2(1,s,0)\underbrace{+a_3\la^2\xi'_3(1,s,0,j)}_{\text{only for line 9}}+b\lambda^2+c\la|\la|\xi'_1(1,s)=0\ ,
$$
for every $\la\ne0$, $s$ even (and $j\le s/2$). This forces $c=0$ by parity in $\la$.

In any case, we must have $a_2\ne0$. Suppose in fact that $a_2=0$. In the cases of lines 8, 10, the identity above would hold for every $i$, and $u_1$ would vanish on every $\tilde S_i$. In the case of line 9, $u_1$ would not depend on $\xi'_2$ and the polynomial $p(\xi'_1,\xi'_3)=u_1(\xi'_1,\xi'_3,1)=a_1{\xi'_1}^2+a_3\xi'_3+b$ would vanish at all points $\big(2s+\kappa,\xi'_3(1,s,0,j)\big)$, for $s$ even and $j\le s/2$. Notice that, for $s$ and $i$ fixed, the values $\xi'_3(1,s,i,j)$ must all be different, because $\xi'_3$ is the only coordinate on $\Sigma_{{\check D}}$ depending on $j$. Then we would have $p=0$ and, by homogeneity, $u_1=0$.

Thus, we have obtained $\tilde u_1=u_1/a_2$ satisfying (a), (b), (c) above.

Assume now that we have constructed $\tilde u_i\in \cP(\bR^{d_0})$, $i=0,\ldots, i_0-1$,
satisfying (a), (b), (c) above.
Consider the polynomial $u_{i_0}$. 
It vanishes on $\tilde S_i$, $i<i_0$, but does not vanish on $\tilde S_i$, $i\geq i_0$.
Hence we can factor out $\tilde u_i$, $i=0,\ldots, i_0-1$, from $u_{i_0}$
and 
there exists a polynomial $q_{i_0}$ such that
$u_{i_0}=q_{i_0} \prod_{i=0}^{i_0-1} \tilde u_i$.
Necessarily $q_{i,0}$ is homogeneous of degree 2 with respect to \eqref {dilation_Rd0},
and vanishes on $\tilde S_{i_0}$
because the polynomials $\tilde u_i$, $i<i_0$, do not vanish on it.
Hence the quotient $q_{i_0}$ will have the form \eqref{u1-810}, resp. \eqref{u1-9}. Arguing as before, it can be shown that $c=0$ and $a_2\ne0$. Then $\tilde u_{i_0}=q_{i_0}/a_2$  has the required properties.
\end{proof}

The higher complexity of the second part of the proof given above was due to the presence of more than one polynomial in $\rho_\fv$, but also by the fact that we did not use explicit formulas for $\xi'_2(1,s,i)$ and $\xi'_3(1,s,i,j)$. To find such formulas does not seem an easy task anyhow, cf. \cite{BR}. On the other hand, the arguments used in the proof emphasise a pattern which is common to all cases at hand.

Note that we have also proved the following identities:
$$
S_m=\begin{cases} \bigcup_{s=0}^{m-1}\{ (|\lambda|(2s+\kappa),\lambda),\ \lambda \in \bR^{\dim\check\fz} \}& \text{(lines 8, 10)}\ ,\\
\bigcup_{i=0}^{m-1}\tilde S_i&\text{(line 9)}\ .\end{cases}
$$

Also note that what prevents $S_m$ from being an algebraic set is the dependence on $|\la|$ of $\xi'_1$. It follows from \eqref{um-1} and \eqref{scaling} that the zero set of $u_m$ in $\bR^d$ is $S_m\cup S_m^-$, where
$$
S_m^-=\big\{(-\xi_1,\xi_2,\xi_3,\la): (\xi_1,\xi_2,\xi_3,\la)\in S_m\big\}\ ,
$$
(with the $\xi'_3$-component omitted for the pairs at lines 8, 10 - this {\it caveat} will not be repeated in the sequel).

\medskip

Let now $G$ be a $\cV_m$-valued, $K$-equivariant Schwartz function $G$ on $\check N$.
Set $f=M_m^* G$. Then $f\in \cS(\check N)^K$ and $f$ belongs to the orthogonal complement of 
$\bigcap_{j=1}^{\nu_m} \ker A^{(m)}_j = \ker U_m$.
Hence the Gelfand transform of $f$ vanishes on $S_m$.
The following lemma justifies that we can choose Schwartz extensions of $\check \cG f$ which vanish on $S_m$.

\begin{proposition}
\label{lem_ext_vanishS_m}
Let $f\in \cS(\check N)^K$ be such that its spherical transform $\check \cG f$ vanishes on $S_m$.
For any $p\in \bN$,
there exists $\psi=\psi^{(p)}\in \cS(\bR^{d_0})$ such that:
\begin{enumerate}
\item[\rm(i)] $(u_m\psi)_{|\Sigma_{{\check D}}}=\check \cG f$
\item[\rm(ii)] there exist $C=C_p>0$ and $q=q(p)$ such that
$\|\psi\|_{(p)}\leq C \|f\|_{(q)}$.
\end{enumerate}
 \end{proposition}
 
 We state first a preliminary lemma.
 \begin{lemma}
\label{lem_factor_out}
Let $P(y)$ be a real polynomial in $y\in \bR^n$.
If  $f(x,y)\in \cS(\bR\times \bR^n)$ vanishes on $\{(P(y),y):y\in\bR^n\}$,
then there exists $\tilde f\in \cS(\bR^{d_0})$
satisfying $f(x,y)=(x-P(y))\tilde f(x,y)$.
Furthermore $\tilde f$ depends linearly and continuously on $f$.
\end{lemma}
\begin{proof}
The conclusion follows easily from Hadamard's lemma (Lemma \ref{hadamard}), once we know that the change of variables $(x,y)\longmapsto \big(x-P(y),y\big)$ preserves $\cS(\bR^{n+1})$ with its topology. This is trivial if $\deg P(y)\le1$. If $\deg P=m>1$, it follows from the inequality
$$
|x-P(y)|+|y|\ge C\big(|x|^{1/m}+|y|\big)\ ,
$$
which can be verified distinguishing between the two cases $|x-P(y)|<|y|$ and $|x-P(y)|\ge|y|$.
\end{proof}

\begin{proof}[Proof of Proposition \ref{lem_ext_vanishS_m}]
Let $\varphi\in \cS(\bR^{d_0})$ be an extension of $\check \cG f$. Such an extension exists by \cite{ADR2}.
Let $P_k$ be the homogeneous component of degree $k$ with respect to \eqref{dilation_Rd0}
 in the Taylor expansion of $\varphi$ around the origin.
Since $\varphi$ vanishes on  $S_m$, which is invariant under these dilations,
$P_k$ vanishes on $S_m$.

By Lemma \ref{lem_claim}, 
there exists $Q_k\in \cP(\bR^{d_0})$ homogeneous of degree $k$ with respect to \eqref{dilation_Rd0}
such that $u_mQ_k=P_{k+2m}$. 

Applying Whitney's extension theorem,
there exists $\psi_1\in C^\infty(\bR^{d_0})$ with compact support around the origin 
and Taylor expansion $\sum_{k\in \bN} Q_k$  at the origin.
Then $\varphi-u_m \psi_1$ vanishes of infinite order at the origin.

We take now a function $\eta$, homogeneous of degree 0 w.r. to the dilations \eqref{dilation_Rd0}, $C^\infty$ away from the origin, and equal to 1 on a conic neighbourhood of $\Sigma_{{\check D}}$ and equal to 0 on a conic neighbourhood of $S_m^-$. Such a function exists because, by the  hypoellipticity of the sublaplacian, $\Sigma_{{\check D}}$ is contained in a conic region around the positive $\xi'_1$-semiaxis, cf. e.g. (15) in \cite{FRY}:
$$
\Sigma_{{\check D}}\subset \big\{ (\xi'_1,\xi'_2,\xi'_3,\la):|\xi'_2|^\half+|\xi'_3|^\half+|\la|\le C\xi'_1\big\}\ .
$$

Then  the function 
$\omega=( \varphi - u_m \psi_1) \eta$ is Schwartz and vanishes on $S_m\cup S_m^-$. By repeated application of Lemma \ref{lem_factor_out}, $\omega=u_m\psi_2$, with $\psi_2$ Schwartz. Take $\psi=\psi_1+\psi_2$. Then
$u_m\psi\eta=\ph\eta$, so that (i) holds.

Consider now the Schwartz norm $\|\psi\|_{(p)}\le \|\psi_1\|_{(p)}+\|\psi_2\|_{(p)}$.

By Lemma \ref{lem_factor_out}, there exist an integer $\nu=\nu(p)\ge p$ and a constant $A_p$ such that
$$
\|\psi_2\|_{(p)}\le A_p\|\omega\|_{(\nu)}\le A'_p\|\varphi - u_m \psi_1\|_{(\nu)}\le A''_p\big(\|\varphi\|_{(\nu)} +\| \psi_1\|_{(\nu+2m)}\big)\ .
$$

In order to estimate $\| \psi_1\|_{(\nu+2m)}$, we use the fact that the Whitney extension of the jet $\{Q_k\}_{k\in\bN}$ can be performed so that the resulting function $\psi_1=\psi_1^{(p)}$ satisfies, for an integer $r=r(p)$ and a constant $B_p$,
$$
\|\psi_1\|_{(\nu+2m)}\le B_p\sum_{k\le r}\|Q_k\|\le B'_p\sum_{k\le r}\|P_{k+2m}\|\le B'_p\|\ph\|_{(r+2m)}\ ,
$$
where the norm of a polynomial is meant as the maximum of its coefficients.

Putting all together,
$$
\|\psi\|_{(p)}\le C_p\|\ph\|_{(\max(r,\nu)+2m)}\ .
$$

By \cite{ADR2}, there are an integer $q=q(p)$ and a constant $C_p$ such that it is possible to choose $\ph=\ph^{(p)}$ above so that 
$$
\|\ph\|_{(\max(r,\nu)+2m)}\le C'_p\|f\|_{(q)}\ ,
$$
and this concludes the proof.
\end{proof}

\medskip

We resume the proof of Proposition \ref{Vm-operator}.

Given $G$, set $f=M_m^* G\in\cS(\check N)^K\in (\ker U_m)^\perp$.
By \eqref{kerUm}, $\check \cG f$ vanishes on $S_m$.

Applying Proposition \ref{lem_ext_vanishS_m}, 
we can choose a Schwartz function $\psi$ such that $u_m\psi$ extends $\check \cG f$.
Defining $h={\check \cG}^{-1}(\psi)$, we easily obtain, on $\Sigma_{{\check D}}$,
$$
\check \cG(U_m h)=u_m \psi =\check \cG f
\ .
$$

This implies
$$
M^*_mM_m h= U_mh=f=M_m^* G\ .
$$

To factor out $M^*_m$, observe that
for any $\lambda\not=0$,
$$
d\pi_\lambda(M_m)^* \pi_\lambda\left( M_mh- H\right)=0\ .
$$

By Proposition \ref{intertwiningPhi}, both sides are 0 when restricted to a subspaces $V(\mu)$ with $\mu$ non-$m$-admissible. If $\mu$ is $m$-admissible, then Proposition \ref{nonvanishing}
implies that $\pi_\lambda\left( M_mh- H\right)=0$ on $V(\mu)$.
Then $M_mh=H$.

It remains to prove the estimates on the Schwartz norms. To the norm estimates given by Proposition \ref{lem_ext_vanishS_m} it is sufficient to add that $M^*_m$ and ${\check \cG}\inv$ are continuous on the appropriate Schwartz spaces. 
For ${\check \cG}\inv$ we refer to \cite{ADR2, FR, FRY}.

\medskip

\vskip1cm

\section{Conclusion}\label{sec_conclusion}

\bigskip

We complete the proof of Theorem \ref{main}.

Let $G\in\big(\cS(\check N)\otimes\cP^k(\fz_0)\big)^K$ as in \eqref{G}.
We decompose $G$ as in \eqref{tilde-decomposition}.
We realise the representation space $\cV_m$ as $W_\alpha$ when $|\al|=m$.
By Lemma~\ref{Vm-operator},
 for each $(\alpha, \beta)$, $[\alpha]+[\beta]=k$,
there exists $h_{\al,\beta}\in \cS(\check N)^K$ such that
$$
\widetilde{p^\al}(v,\zeta)\tilde g_{\al\beta}=M_{\al,\zeta} h_{\al,\beta}
\ ,
$$
where the operator $M_{\al,\zeta}=\sum_{j=1}^{\nu_m}A^{(m)}_j b_j^{(\al)}(\zeta)
$ is the realisation of $M_m$ on $W_\alpha$.

In the notation of \eqref{expansion3}, the operators $\tilde D^{\al''}_\zeta$ form a basis of $\big(\bD(\check N)\otimes \cP(\fz_0)\big)^K$. Therefore,
each $M_{\al,\zeta}$ can be expressed as a linear combination of the $\tilde D^{\al''}_\zeta$ with $[\alpha'']=k$,
and one can write $G$ as 
$$
G=\sum_{[\al'']= k}\tilde D_\zeta^{\al''} H_{\al''}\ , 
$$
where the functions $H_{\al''}$ are finite linear combinations of $h_{\al,\beta}$. 

The norm estimates are obvious after Proposition \ref{Vm-operator}.

\vskip2ex

\noindent
{\bf Acknowledgments.} Parts of this work were carried out during the third author's
stay at the Max-Planck-Institut f\"ur Mathematik (Bonn) and 
Centro di Ricerca Matematica Ennio De Giorgi (SNS, Pisa). 
She would like to thank these 
institutions for warm hospitality and support.

\end{document}